\documentclass[12pt, twoside]{article}
\usepackage{amsmath,amsthm,amssymb}
\usepackage{times}
\usepackage{enumerate}

\def\titlerunning#1{\gdef\titrun{#1}}
\makeatletter
\def\author#1{\gdef\autrun{\def\and{\unskip, }#1}\gdef\@author{#1}}
\def\address#1{{\def\and{\\\hspace*{18pt}}\renewcommand{\thefootnote}{}%
\footnote {#1}}%
\markboth{\autrun}{\titrun}}
\makeatother
\def\email#1{e-mail: #1}
\def\subjclass#1{{\renewcommand{\thefootnote}{}%
\footnote{\emph{Mathematics Subject Classification (2010):} #1}}}
\def\keywords#1{\par\medskip
\noindent\textbf{Keywords.} #1}

\newtheorem{thm}{Theorem}[section]
\newtheorem{cor}[thm]{Corollary}
\newtheorem{lem}[thm]{Lemma}
\newtheorem{prop}[thm]{Proposition}
\newtheorem{facts}[thm]{Facts}

\theoremstyle{definition}
\newtheorem{defin}[thm]{Definition}
\newtheorem{rem}[thm]{Remark}
\newtheorem{exa}[thm]{Example}

\numberwithin{equation}{section}

\newcommand\Spin{\operatorname{Spin}}
\newcommand\Ad{\operatorname{Ad}}
\newcommand\Exp{\operatorname{Exp}}
\newcommand\diag{\operatorname{diag}}
\newcommand\trace{\operatorname{trace}}
\newcommand\Id{\operatorname{Id}}
\newcommand\SO{\operatorname{SO}}
\newcommand\SU{\operatorname{SU}}
\newcommand\so{\mathfrak{so}}

\begin{document}


\baselineskip=17pt


\titlerunning{Compact homogeneous Riemannian manifolds with low co-index of symmetry}

\title{Compact homogeneous Riemannian manifolds with low co-index of symmetry}

\author{J\"{u}rgen Berndt
\and 
Carlos Olmos
\and
Silvio Reggiani}

\date{}

\maketitle

\address{J. Berndt: Department of Mathematics, King's College London, London, WC2R 2LS,
United Kingdom; \email{jurgen.berndt@kcl.ac.uk}
\and
C. Olmos: Facultad de Matem\'atica, Astronom\'ia y F\'isica, Universidad Nacional de C\'ordoba, 
Ciudad Universitaria, 5000 C\'ordoba, Argentina; \email{olmos@famaf.unc.edu.ar}
\and
S. Reggiani: Dpto.\ de Matem\'atica, ECEN-FCEIA, Universidad Nacional de Rosario, Av.\ Pellegrini 250, 2000 Rosario, Argentina; \email{reggiani@fceia.unr.edu.ar}}



\begin{abstract}
We develop a general structure theory for compact homogeneous Riemannian
manifolds in relation to the co-index of symmetry. We will then use these
results to classify irreducible, simply connected, compact homogeneous
Riemannian manifolds whose co-index of symmetry is less or equal than three. We
will also construct many examples which arise from the theory of polars and
centrioles in Riemannian symmetric spaces of compact type. 

\end{abstract}

\section {Introduction}

A homogeneous manifold is a manifold $M$ together with a Lie group $G$ acting
transitively on $M$. Homogeneous manifolds are of particular interest in
geometry, topology, algebra and physics. In the context of Riemannian geometry
one is interested in homogeneous Riemannian manifolds, where the group $G$ acts
transitively by isometries. Killing fields are vector fields preserving the
metric on the manifold. Such vector fields are of interest in particle physics
where they correspond to symmetries in theoretical models. On a homogeneous
Riemannian manifold there are many Killing vector fields. More precisely, a
connected complete Riemannian manifold $M$ is homogeneous if and only if at
every point $p \in M$ and for every $v \in T_pM$ there exists a Killing field $X$ on $M$ with $X_p = v$. This characterization of homogeneous Riemannian manifolds is very useful.

A Killing field is uniquely determined by its value and its covariant derivative
at a point. Important classes of homogeneous Riemannian manifolds are obtained
by imposing additional conditions on the covariant derivative of its Killing
fields. For example, a homogeneous Riemannian manifold $M$ is a Riemannian
symmetric space if and only if for every point $p \in M$ and  every 
$v\in T_pM$, there exists a Killing
field $X$ on $M$ with $X_p = v$ and $(\nabla X)_p = 0$. Riemannian symmetric
spaces were classified by \'{E}lie Cartan and there is a beautiful theory
relating such spaces to the algebraic theory of semisimple Lie algebras (see
e.g.\ \cite{Helgason}).

Motivated by this characterization of symmetric spaces, the second and third
author together with Tamaru  introduced in  \cite{OlmosReggianiTamaru} the index of symmetry of
a Riemannian manifold. Let $M$ be a Riemannian manifold and denote by $\mathfrak
K(M)$ the Lie algebra of Killing fields on $M$. 
For $q \in M$ define the symmetric subspace $\mathfrak s_q$ of $T_qM$ by
$\mathfrak s_q = \{X_q \in T_qM : X \in \mathfrak K(M) \ {\rm and}\ (\nabla X)_q
= 0\}$. The index of symmetry $i_{\mathfrak s}(M)$ of $M$ is defined as
$i_{\mathfrak s}(M) = \inf \{\dim (\mathfrak s_q) : q \in M\}$, and the
co-index of symmetry $ci_{\mathfrak s}(M)$ is defined by $ci_{\mathfrak s}(M) =
\dim(M) - i_{\mathfrak s}(M)$. If $M$ is a homogeneous Riemannian manifold, say
$M = G/H$, then the symmetric subspaces form a $G$-invariant distribution
$\mathfrak s$ on $M$. This distribution is called the distribution of symmetry
on $M$. In \cite{OlmosReggianiTamaru} it was shown that the distribution of symmetry is
integrable and its maximal integral manifolds are Riemannian symmetric spaces
which are embedded in $M$ as totally geodesic submanifolds. For normal
homogeneous Riemannian manifolds and a class of naturally reductive homogeneous
Riemannian manifolds the distribution of symmetry was explicitly determined in
\cite{OlmosReggianiTamaru}.

As mentioned above, a homogeneous Riemannian manifold is a Riemannian symmetric
space if and only if $ci_{\mathfrak s}(M) = 0$. Thus the co-index of symmetry
can be regarded as a measure for how far a homogeneous Riemannian manifold fails
to be a Riemannian symmetric space. The purpose of this paper is to develop some
general structure theory for compact homogeneous Riemannian manifolds in
relation to the co-index of symmetry. We will then use these results to classify
irreducible, simply connected, compact homogeneous Riemannian manifolds whose
co-index of symmetry is less or equal than $3$. We will also determine the
co-index of symmetry for  compact homogeneous Riemannian manifolds which arise
as total spaces over polars in Riemannian symmetric spaces of compact type and
whose fibers are centrioles. 

The paper is organized as follows. In Section 2 we present some basic results
about Riemannian symmetric spaces and which will be used later.

In Section 3 we investigate $G$-invariant autoparallel distributions $\mathcal
D$ on compact homogeneous Riemannian manifolds $M = G/H$. Such a distribution is
said to be strongly symmetric with respect to $G$ if every maximal integral
manifold $L$ of $\mathcal D$ is a Riemannian symmetric space and the
transvection group of $L$ is contained in $\{g_{\vert L} : g \in G\ {\rm and}\
g(L) = L\}$.
The main result is Theorem \ref{main} which says, roughly speaking, that if the
co-rank $k$ of a strongly symmetric $G$-invariant distribution on $G/H$
satisfies $k \geq 2$, then $M$ is a homogeneous space of a normal semisimple
subgroup $G'$ of $G$ with $2\dim(G') \leq k(k+1)$.

In Section 4 we introduce the index and co-index of symmetry and review some
results from \cite{OlmosReggianiTamaru}.

In Section 5 we develop some general structure theory for compact homogeneous
Riemannian manifolds in relation to the co-index of symmetry. The main result in
this section is Theorem \ref{mainco}:
Let $M$ be a simply connected compact homogeneous Riemannian manifold and
assume that $M$ does not split off a symmetric de Rham factor. 
Then $k = ci_{\mathfrak s}(M) \geq 2$ and there exists a transitive semisimple
normal Lie subgroup 
$G'$ of  the isometry group of $M$ 
 such that $2\dim (G') \leq  k(k+1)$. The
 equality holds if and only if 
the universal covering group 
of $G'$ is $\Spin (k+1)$. Moreover, if the equality holds and  
$ci_{\mathfrak s}(M) \geq 3$, then the isotropy group of $G'$ has positive
dimension. 

In Section 6 we investigate compact homogeneous Riemannian manifolds with
$ci_{\mathfrak s}(M)$ $= 3$. We will construct explicitly a $2$-parameter family
of non-homothetical $SO(4)$-invariant Riemannian metrics on  $M = SO(4)/SO(2)$
with $ci_{\mathfrak s}(M) = 3$.
The main result is Theorem \ref{classification} and states that every
irreducible, simply connected, compact homogeneous Riemannian manifold with
$ci_{\mathfrak s}(M) = 3$ is homothetic to $M = SO(4)/SO(2)$ with such a
Riemannian metric.

In Section 7 we investigate compact homogeneous Riemannian manifolds with
$ci_{\mathfrak s}(M)$ $= 2$. We will construct explicitly two $1$-parameter
families of non-homo\-thetical left-invariant Riemannian metrics on  $M =
\Spin(3)$ with $ci_{\mathfrak s}(M) = 2$.
The main result is Theorem \ref{dos} and states that every irreducible, simply
connected, compact homogeneous Riemannian manifolds with $ci_{\mathfrak s}(M) =
2$ is homothetic to $M = \Spin(3)$ with such a left-invariant Riemannian metric.

In Section 8 we review the construction by Nagano and Tanaka in \cite{NaganoTanaka} of
certain fibrations $K^+/K^{++} \to K/K^{++} \to K/K^+$. Let $M = G/K$ be a
simply connected Riemannian symmetric space of compact type such that $K$ is the
isotropy group of $G$ at $o \in M$. Let $p$ be an antipodal point of $o$ in $M$.
Then the orbit $B = K \cdot p = K/K^+$ of $K$ through $p$ is a so-called polar
of $M$. Assume that $\dim(B) > 0$ and that $B$ is irreducible. Let $q$ be the
midpoint of a distance minimizing geodesic joining $o$ and $p$ and assume that
the orbit $S = K \cdot q = K/K^{++}$ is not a Riemannian symmetric space with
respect to the induced metric from $M$. The fibers $K^+/K^{++}$ of the fibration
$K/K^{++} \to K/K^+$ are centrioles in $M$. We will show in Theorem \ref{polar}
that the co-index of symmetry of the orbit $S = K/K^{++}$, with the induced
Riemannian metric, is equal to the dimension of  the polar $B = K/K^+$. This
provides many examples of compact homogeneous Riemannian manifolds for which the
co-index of symmetry can be calculated explicitly in a rather simple way.

\section {Preliminaries and basic results}

Let $M= G/K$ be an $n$-dimensional, connected, simply connected, Riemannian
symmetric 
space, where $n \geq 2$ and $(G,K)$ is an effective symmetric pair. 
We denote by $I(M)$ the full isometry group of $M$ and by $I^o(M)$ the connected
component of $I(M)$ containing the identity transformation of $M$. Note that $G
= I^o(M)$ if the Riemannian universal covering space of 
$M$ has no Euclidean de Rham factor, or equivalently, if $M$ is a semisimple
Riemannian symmetric space.
The geodesic symmetry  at $p\in M$ will be denoted by $\sigma _p$. A Riemannian
symmetric space is said to be inner if $\sigma_p \in I^o(M)$ for one (and hence
for all) $p \in M$.

\begin{lem} \label {translates}
Let $\gamma \in Z_{I(M)}(G)$ be in the centralizer of $G$ in $I(M)$ and assume
that for every $q\in M$ with $\gamma(q) \neq q$ the isometry $\gamma$ translates
a minimizing geodesic in $M$ joining $q$ and $\gamma (q)$. 
Then we have $  \sigma _p \gamma \sigma ^{-1} _p = \gamma^{-1} $ for all 
$p\in M$. If, in particular, $M$ is inner, then $\gamma^2 = {\rm id}_M$.
\end{lem}

\begin {proof}
Let $p\in M$ and put $\bar {\gamma} = 
\sigma _p \gamma \sigma ^{-1} _p$. It is clear that 
$\bar {\gamma}$ and $\bar {\gamma}^{-1}$
satisfy the assumptions of this lemma.
Let $q\in M$ with $\gamma(q) \neq q$. By assumption, there exists a geodesic
$\beta : {\mathbb R} \to M$ 
through $q$ and $\bar{\gamma}(q)$ which minimizes the distance between
$q=\beta (0)$ and $\bar {\gamma}(q) = \beta (a)$ with $a > 0$
and is translated by $\bar {\gamma}$. 
Then  $\bar {\gamma}(\beta (t)) = \beta (t + a)$ 
and $\bar {\gamma}^{-1} (\beta (t)) = \beta (t -a)$ for all $t \in {\mathbb
R}$. 
Since $\gamma \in Z_{I(M)}(G)$ and $\sigma_q\sigma_p \in G$, we have $\gamma(q) =
(\sigma_q\sigma_p)\gamma(\sigma_q\sigma_p)^{-1}(q)$ and therefore
\[
\gamma(q) = \sigma_q\bar{\gamma}(q) = \sigma_q\beta(a) = \beta(-a) =
\bar{\gamma}^{-1}(\beta(0))
= \bar{\gamma}^{-1}(q) = \sigma_p\gamma^{-1}\sigma_p^{-1}(q),
\]
which implies $  \sigma _p \gamma \sigma ^{-1} _p = \gamma^{-1} $. 
\end {proof}

\begin{rem} \label {Wolf} A well-known result of Joseph Wolf  states that
in a homogeneous Riemannian manifold $N$ any geodesic loop must be a closed 
geodesic. In fact, let
$\beta : \mathbb R \to N$ be a unit-speed geodesic and let 
$X$ be a Killing field on 
$N$ with
$X(\beta (0)) = \beta ' (0)$. Then it follows from the Killing equation that 
the inner product between $X(\beta (t))$ and $\beta ' (t)$ is a constant
function. The value 
of the inner product at $t = 0$ is equal to $1$. 
Assume that $\beta (0) = \beta (a)$ with some $a\neq 0$. Then the inner product
between
$X(\beta (a)) = X(\beta (0)) = \beta'(0)$ and $\beta ' (a)$ is equal to $1$, and
it follows
from the Cauchy-Schwartz inequality that $\beta '(0)= \beta '(a)$, which shows
that 
$\beta $ is a closed geodesic. 
\end{rem}

\begin {cor}\label {globally}
Let $M=G/K$ be a Riemannian globally symmetric space, 
where $(G,K)$ is an effective symmetric pair. Let 
 $\pi : M \to N= G/ \bar {K}$ be a $G$-equivariant local 
isometry, where the action of $G$  on $N$ is almost effective.   
Then $N$ is a Riemannian globally symmetric space. 
\end {cor}

\begin {proof}
 Let $\Gamma \subset I(M)$ be the group of deck transformations of 
 $N$. We can assume that $\pi$ is not a global isometry, or equivalently, that
$\Gamma$ is non-trivial. Since the action of $G$ on $M$ projects to an action on
$N$, $M$ is connected and $\Gamma$ is discrete, it follows that $G$ normalizes
$\Gamma$. Let ${\rm id}_M \neq \gamma \in \Gamma$, 
 $q\in M$, and let 
$\beta : \mathbb {R} \to M$ be any minimizing 
 geodesic between $\beta (0) 
= q$ and $\beta (a) = \gamma (q)$. 
The geodesics  $\gamma ( \beta (t))$ and $\beta (t)$ in $M$
project 
to the same geodesic $\bar {\beta } (t) = 
\pi (\beta (t)) = \pi (\gamma(\beta (t)))$ in $N$. Since $\bar {\beta} (0) =
\bar { \beta} (a)$, 
it follows from Remark \ref {Wolf}  that the geodesic $\bar {\beta}$ in $N$ is
periodic with period $a$ 
(not necessarily the smallest period). 
This implies that $\gamma (\beta (t)) = \beta (t + a)$ and so 
$\gamma $ translates the geodesic $\beta$. From Lemma 
\ref 
{translates} we get $\sigma _p \gamma \sigma ^{-1} _p 
= \gamma ^{-1}$ for all 
$p\in M$ and $\gamma \in \Gamma$. This implies that the geodesic symmetry
$\sigma _p$ on $M$ descends to a geodesic symmetry 
of $N$. We now conclude that $N$ is globally symmetric.
\end {proof}

\begin{rem}
Conjugation by $\sigma_p$ defines a group automorphism of $\Gamma$, and the
proof of Corollary \ref{globally} shows that this automorphism is given by
$\gamma \mapsto \gamma^{-1}$. This implies that  $\Gamma $ is an abelian group,
which reflects the well-known fact that the homotopy group  of a globally
symmetric space is an abelian group. From Lemma \ref{translates} we also see
that $\Gamma$ must be isomorphic to a direct product ${\mathbb Z}_2 \times
\ldots \times {\mathbb Z}_2$ if $M$ is an inner Riemannian symmetric space.
\end{rem}

\begin{rem}
In this remark we fill a gap in the proof of Lemma 5 on page 491 of \cite {EschenburgOlmos}
for the global symmetry and simplify the arguments. 
In fact, condition (*) in  \cite [page 493] {EschenburgOlmos}
needs not to be  true a priori, since the equality only holds 
for the restriction of those groups to the flat.
Let us keep the notation of \cite{EschenburgOlmos} and prove Lemma 5. 

For any maximal flat $F$ in the globally symmetric space $X$ let
$\tau _F$ be the abelian subgroup of $I_0(X)$ which 
consists of the glide translations along geodesics in $F$. More precisely, 
$\tau _F = \{\text {Exp}(X): X \in \mathfrak p \}$, 
where $\mathfrak p$ 
is the Cartan subspace at some point $p\in F$. The abelian subgroup 
 $\tau _F$ is a normal subgroup of $I_F(M)$, 
the subgroup of $I(X)$ which
leaves $F$ invariant. Since any element of 
$\Gamma _F \subset \Gamma$ acts as a translation 
on $F$ (Sublemma 1 is correct!) it follows that $\tau _F$ commutes with 
$\Gamma _F$. 
In fact, for all $g\in \Gamma _F $ and $X\in \mathfrak p$ we have
$g\text {Exp}(X)g^{-1} = \text {Exp}(g_*(X))
\in \tau _F$. Since $g$ restricted to $F$ is a translation, this implies
$g_*(X) = X$.  

From the assumptions of Lemma 5 one obtains that 
$\{\tilde g \tau _F \tilde g ^{-1}: \tilde g \in \tilde G\}$ 
contains any geometric transvection subgroup 
$\{\text {Exp}
(tX):t\in \mathbb {R}\}$ where $X$ belongs to any Cartan subspace. 
Then $\tilde {G}$ and $\tau _F$ generate $T$, the full transvection group of 
$X$. Since $\tilde {G}$ and $\tau _F$ commute with any element of 
$\Gamma _F$ we conclude that $T$ commutes with $\Gamma _F$. Since, as stated in 
Sublemma 1, $\Gamma$ is the union of $\Gamma _F$, $F$ an arbitrary flat, we
obtain 
that $T$ commutes with $\Gamma$. 
Since the geodesics in $M$ have no self-intersection (since they lie in 
a globally symmetric 1-1 immersed flat), we have that any $\gamma \in \Gamma$
satisfies the 
assumptions of Lemma \ref {translates}. 
Then $\sigma _p (\Gamma) = \Gamma$ and so the geodesic symmetry $\sigma _p$ 
descends from $X$ to the quotient 
$M$, which implies that $M$ is globally symmetric. 
This completes the proof of Lemma 5 on page 491 of \cite {EschenburgOlmos}.
\end{rem}

Let $M = G/K$ be a connected, simply connected, Riemannian symmetric space, 
where $(G,K)$ is an effective  symmetric pair. Denote by ${\mathfrak g}$ the Lie
algebra of $G$. Assume that every Killing field  $X$ on $M$, $X\in \mathfrak g$,
is bounded. 
This is equivalent to saying that the de Rham decomposition of $M$ does not
contain a Riemannian symmetric space of noncompact type. 
Let  $M = \mathbb R ^k \times M_1 \times  ... \times M_r$ 
be the de Rham decomposition of $M$ ($k = 0$ is possible)
and let us write 
$$G/K = \mathbb R^k \times (G_1/K_1) \times ... \times (G_r/K_r).$$ 
where $M_i = G_i/K_i$ is a connected, simply connected, Riemannian symmetric
space of compact type.
If $M_i$ is not of group type, then $G_i$ is a compact simple Lie group. 
If $M_i$ is of 
group type then $G_i = \bar {G}_i\times \bar {G}_i$ 
where $\bar {G}_i$ is a compact simple
Lie group and $K_i = \text {diag} (\bar {G}_i\times \bar {G}_i)$. 
Moreover, $M_i \simeq \bar {G}_i$. 

Choose $p= (p_0, \ldots , p_r) \in M$ so that $K = G_p$ is the isotropy subgroup
of $G$ at $p$. 
Then the isotropy representation of 
$K$ on $T_pM$ is, up to the trivial representation on 
$\mathbb R ^k$, 
 the direct sum of the irreducible representations of 
$K_i$ on $T_{p_i}M_i$.

\begin{defin} The Lie algebra $\mathfrak g_i$ 
of $G_i$ ($i=1, \ldots , r$), considered as a subalgebra of ${\mathfrak g}$,
is called a {\it symmetric irreducible factor} 
of $\mathfrak g$.
\end{defin}

Note that a symmetric irreducible factor of $\mathfrak g$ 
is either a  simple Lie algebra or the direct sum of a simple Lie algebra with
itself.

Let  $N=G/\bar {K}$ be a  Riemannian symmetric space which is 
not necessarily simply 
connected. We assume that $N$ is equivariantly 
covered by $M=G/K$; see Corollary 
\ref {globally}.
Then the autoparallel distributions on $M$ corrresponding to the factors in the
de Rham decomposition of $M$ induce autoparallel distributions on $N$. In fact,
any element $\gamma$ 
in the group $\Gamma$ of deck 
transformations of the projection $\pi : M \to N$ commutes with the transvection
group $G$ of $M$. 
This implies that $\gamma$ preserves the  autoparallel distribution 
 on $M$ associated to any of its de Rham factors.  
If $\bar K^o$ is the connected component of $\bar K$, then, as for the simply 
connected case, the isotropy representation of $\bar K^o$ decomposes, 
up to a trivial representation, as a direct sum of irreducible representations.

The following lemma is easy to prove.

\begin{lem}\label {irreduciblesym} Let $N=G/\bar {K}$ be a  
Riemannian globally symmetric space, where $G$ is the group of 
transvections ($N$ is not assumed to be  simply connected). 
Let 
$\tilde {\mathfrak g} '$ be an ideal of 
$\mathfrak g$ that contains the 
abelian part of $\mathfrak g $. Assume 
 that $\tilde {G}'$ does not act transitively on $N$, where 
$\tilde {G}'$ is the normal Lie subgroup of $G$ with Lie algebra
$\tilde {\mathfrak g} '$. Let $\hat {\mathfrak g}$ be a complementary 
ideal to $\tilde {\mathfrak g} '$.  
Then $\hat {\mathfrak g}$ contains an irreducible symmetric 
factor $\mathfrak g _{i}$ of $\mathfrak g$.
\end{lem}

\begin{rem} \label {remark} 
If in the situation of Lemma \ref{irreduciblesym} 
the symmetric space $N$ is simply connected,
and if $\hat G$ contains only one of the two factors of $G_i = \bar G_i\times 
\bar G_i$, 
where  $M_i$ is a de Rham factor  of group type, 
then $\hat G/\hat G_p$ is not a 
symmetric presentation of the symmetric orbit $\hat G \cdot p$, $p\in N$. 
\end{rem}

\section {Symmetric autoparallel distributions}

Let $M=G/H$ be an $n$-dimensional compact homogeneous Riemannian manifold, where
$n \geq 2$ and
$G$ is a connected Lie subgroup of $I(M)$.  Let 
$\mathcal {D}$ 
be an autoparallel $G$-invariant distribution on 
$M$ of rank $r > 0$. We denote by $k = n-r = \dim(M) - {\rm rk}({\mathcal D})$
the corank of ${\mathcal D}$. The  maximal integral manifold of $\mathcal {D}$
containing a point $p \in M$ 
will be denoted by $L(p)$. Note that $L(p)$ is a totally geodesic 
submanifold of $M$ since ${\mathcal D}$ is autoparallel. For all $g \in G$ and
$p \in M$ such that $g(L(p)) = L(p)$ we denote by $g_{\vert L(p)}$ the isometry
on $L(p)$ which is obtained by restricting $g$ to $L(p)$. If $X$ is a Killing
field on $M$ which is induced by $G$, then we denote by $X_{\vert L(p)}$ the
restriction of $X$ to $L(p)$.

\begin{defin}
The $G$-invariant autoparallel
distribution $\mathcal D$ is  {\it strongly symmetric} 
with respect 
to $G$ if every integral manifold $L(p)$ of $\mathcal {D}$ is a 
globally symmetric space and the identity component of
$\{g_{\vert L(p)}: g(L(p)) = L(p),\ g \in G\}$ 
contains the 
transvection group of $L(p)$  (or equivalently, since the Killing 
fields associated to $G$ are bounded, coincides with the 
transvection group of $L(p)$).
\end{defin}

From Corollary \ref {globally} one has the following  
equivalent definition: 

\begin{defin} 
The $G$-invariant autoparallel 
distribution $\mathcal D$ is 
{\it strongly symmetric} with respect 
to $G$ if for every $p \in M$ and every $v\in {\mathcal D}_p$ there exists a 
Killing field $X$ on $M$ which is induced by 
$G$ such that $X_p = v$ and  $X_{\vert L(p)}$ 
is parallel at $p$.
\end{defin}

Let $M = G/H$ be a compact homogeneous Riemannian manifold and let 
$\mathcal D$ be a non-trivial $G$-invariant distribution 
on $M$ which is strongly symmetric with respect to $G$. 
Let $\mathfrak g = T_eG$ be the Lie algebra of $G$, where any  element 
$X$ 
of $\mathfrak g$ is identified with  the Killing field $p \mapsto X.p
= \frac {d}{dt}_{\vert t=0}\text {Exp}(tX)(p)$ of $M$.
It is important to note that with 
 this identification the brackets change sign, since the Killing field 
$X$ is related via the isometry $g$ with the right-invariant vector field of $G$
with initial condition $X\in T_eG$. 
 The subspace
$$\mathfrak g ^{\mathcal D} : = \{X\in \mathfrak g : 
X \text { lies on }\mathcal D\}$$
of ${\mathfrak g}$  is an ideal of $\mathfrak g$ 
since $\mathcal D$ is $G$-invariant.

\begin{lem} \label {premain} Let 
$\mathfrak g '\subset \mathfrak g$ be a complementary 
ideal of $\mathfrak g ^{\mathcal D}$ and 
let $G'$ be the normal subgroup of $G$ with Lie algebra $\mathfrak g'$.
Then $2\dim(G') \leq  k (k + 1)$, where 
$k = n - r = \dim (M) - {\rm rk} (\mathcal D)$ is the corank of $\mathcal D$.
Moreover, for $k\geq 2$, the equality holds if and only if the universal
covering group 
of $G'$ is $\Spin (k+1)$. For $k=1$,  
$\mathcal {D}$ is a parallel distribution and the Riemannian universal covering
space of $M$ 
splits off a line (and then $M$ is locally symmetric).  
\end{lem}

\begin {proof} Since the integral manifolds 
of $\mathcal  D$ are not necessarily
closed submanifolds of $M$ 
(when they have a Euclidean local factor), 
we will consider, locally, the quotient 
space of  $M$ by the foliation  given by the maximal
integral manifolds of $\mathcal D$. Let $p\in M$ and let 
$U$ 
be the domain of a Frobenius chart of   the
the distribution $\mathcal D$ in a neighborhood of $p$. Let us denote by
$\mathcal F$ 
the foliation of $U$ given by the maximal integral manifolds of 
$\mathcal D _{\vert U}$ and by  $\bar U = U/\mathcal F$ the quotient space. Let
$\pi: U\to \bar U$  
be the canonical projection. 
 Any $Z \in  {\mathfrak g}$, 
regarded as a Killing field on $U$, projects via $\pi$ to a vector field 
$\bar {Z}$ on $\bar U$, since any $g\in G$ which is close to the identity 
preserves (locally) the 
foliation $\mathcal F$. We have that $\bar Z = 0$ if and only if 
$Z_{\vert U}$ is tangent to the distribution
$\mathcal D _{\vert U}$. 

Let $p\in U$ be fixed and let $q=g(p)\in U$. Since $\mathcal D$ is 
$G$-invariant, we have $Z_q \in \mathcal D _q$ if and only if 
$\Ad (g)(Z)_p \in \mathcal D _p$. Let  $\Omega$ be a neighbourhood 
of the identity $e$ in $G$ such that $\{ g(p) : g \in \Omega\} \subset U$. Then, if $\bar Z
=0$, 
we have $\Ad (g)(Z)_p 
\in \mathcal D _p$ for all $g\in \Omega$. Since $\Omega$ generates 
$G$ this gives $\Ad (g)(Z)_p 
\in \mathcal D _p$ for all $g\in G$. This implies that 
the Killing field 
$Z$ is tangent to $\mathcal D$.  
Therefore, $Z \in \mathfrak g ^{\mathcal {D}}$ if and only if $\bar Z =0$.

Let us now consider the normal homogeneous Riemannian metric on
$M = G/H$. This metric, 
restricted to $U$, projects via $\pi$ to a Riemannian metric on $\bar U$. 
In fact, if $p\in U$, 
$\bar U$ can be locally regarded as $G/\tilde H$, 
where $\tilde H \supset H $ is the Lie subgroup 
of $G$ which leaves invariant $L(p)$. So any 
element $Z \neq 0$ in the complementary ideal $\mathfrak g '$ 
of $\mathfrak g ^{\mathcal D}$ can be regarded as 
a non-trivial Killing field on $\bar U$. 
If $\bar p = \pi (p)$, then 
the map $j: \mathfrak g '\to 
T_{\bar p}\bar U \times \so (T_{\bar p }\bar U)$,
$j(Z) = (\bar Z _{\bar p}, 
(\nabla \bar Z)_{\bar p})$, which 
assigns to $Z$ the initial conditions  of the Killing field 
$\bar Z$ at $\bar p$, is injective. Then,  since  
$k = \dim(\bar U)$, we conclude that
$2\dim(G')\leq k(k+1)$.  

We now consider the injective Lie algebra homomorphism 
$\pi_*: \mathfrak g ' \to \mathcal K (\bar U), \ 
Z \mapsto \bar Z$,
where $\mathcal K (\bar U)$ denotes the Lie algebra  of Killing fields on 
$\bar U$ with the projection of  the normal homogeneous metric on $M$ and
where the bracket 
on $\mathfrak g '$ is the bracket of Killing fields. 
Note that   
$2\dim (\mathfrak g ') \leq 2\dim (\mathcal {K}(\bar U))\leq k(k+1)$.  
It follows that $\bar U$ has constant curvature when    $2\dim (\mathfrak g ') 
=  k(k+1)$.  In this case, since $\mathfrak g '$ 
admits a bi-invariant metric, 
we get $\mathfrak g ' \simeq \mathcal {K}(\bar U)\simeq \mathfrak {so}(k+1)$. 
Then the universal covering group of $G'$ is $\Spin (k+1)$ if $k\geq 2$.

For $k=1$ we have $\dim (G') = 1$. If $G^{\mathcal D}$ is the normal
Lie subgroup 
of $G$ with Lie algebra $\mathfrak g ^{\mathcal D}$, then 
the $G^{\mathcal D}$-orbits in $M$ coincide with the integral manifolds 
$L(q)$ of $\mathcal D$. In fact, $G^{\mathcal D}$ cannot be transitive on $M$ 
since the orbit $G^{\mathcal D} \cdot q$ is contained in $L(q)$ for all $q\in
M$. 
Therefore, since $G$ is transitive 
on $M$ and $\dim (\mathfrak g ^{\mathcal D}) = \dim (\mathfrak g) -1$, 
any $G^{\mathcal D}$-orbit must have codimension one and therefore 
$G^{\mathcal D} \cdot q = L(q)$. 
Thus $G^{\mathcal D}$ acts on $M$ 
with cohomogeneity one and without singular orbits. 
In fact, since 
$G^{\mathcal D}$ is a normal subgroup of $G$, we have $G^{\mathcal D}\cdot g(q)
= g(G^{\mathcal D} \cdot q)$ for all 
$g\in G$. Then the $1$-dimensional distribution 
perpendicular to the 
$G^{\mathcal D}$-orbits (or equivalently, orthogonal 
to $\mathcal D$),  is autoparallel. 
Since $\mathcal D$ is also an autoparallel distribution, 
then both distributions must be parallel.  
This implies that the Riemannian universal 
covering space of $M$ splits off a line. 
\end {proof}

\begin{rem} \label {effectively}
The normal subgroup $G^{\mathcal D}$ of $G$ with Lie algebra 
$\mathfrak g ^{\mathcal D}$ 
acts effectively on every integral manifold $L(q)$ of 
${\mathcal D}$. In fact,  as in Lemma \ref {premain}, let $G'$
be the normal Lie subgroup of $G$ associated with a complementary 
ideal of $\mathfrak g ^{\mathcal D}$. This gives an almost direct product 
$G= G^{\mathcal D}\times G'$. 
Since $G$ is transitive on $M$, the subgroup $G'$ acts transitively 
on the family $\{L(q):q\in M\}$.  Let $g\in G^{\mathcal D}$ and 
$p\in M$  such that 
$g_{\vert L(p)} : L(p) \to L(p)$ is the identity, and let $L(q)$ be another 
maximal integral manifold of $\mathcal D$. Then there exists $g'\in G'$ such that 
$g'(L(p)) = L(q)$. Let $q'= g'(p') \in L(q)$ with $p' \in L(p)$. Then 
$g(q') = g(g'(p'))= g'(g(p'))= g'(p') = q'$, and thus $g=e$. 
\end{rem}

We continue with the notations and  assumptions of Lemma \ref{premain}. 
Let $q \in M$ and define 
$$\bar G^q = \{ g_{\vert L(q)}: g(L(q))= L(q),\, g\in G\}^o,$$
which coincides with the transvection group of $L(q)$ since 
$\mathcal D$ is strongly symmetric with respect to $G$. 
Let $G_q$ be the isotropy group of $G$ at $q$ and define
$$\bar K^q = \{g_{\vert L(q)}: g\in G_q\}.$$
Then $\bar G^q/\bar K^q$ is a symmetric presentation of the symmetric 
space $L(q)$. Note that $G_q$ and hence $\bar K^q$ are connected if
$M$ is simply connected. 

The subgroup
$$\bar G'^q = \{ g_{\vert L(q)}: g(L(q))= L(q),\, g\in G'\}^o$$
is a normal subgroup of $\bar G^q$ 
and commutes with $G^{\mathcal D}$ and 
$\bar G^q = \{ g_1g_2 : g_1 \in \bar G'^q, g_2 \in G^{\mathcal D}\}$, 
where $G^{\mathcal D}$ is 
identified with its restriction to  $L(q)$. 
In general $\bar G'^q$ and  $G^{\mathcal D}$  intersect in a 
normal subgroup of $\bar G'^q$  with positive dimension. Let 
$\bar {\mathfrak  g}'^q$ be the Lie algebra of $\bar G'^q$ and define 
$\mathfrak u = \bar {\mathfrak  g}'^q\cap \mathfrak g ^ {\mathcal D}$.
Let $\hat {\mathfrak g}$ be a complementary ideal to $\mathfrak u$ in 
$\mathfrak g ^ {\mathcal D}$. Note that $\hat {\mathfrak g} $ 
is an ideal of the Lie algebra $\bar {\mathfrak  g} ^q$ of $\bar G^q$ which
can be identified with 
an ideal of $\mathfrak g$ which does not depend on the choice of $q\in M$. 
If $\hat G \subset G^{\mathcal D}$ denotes the normal Lie 
subgroup of $G$ associated 
with $\hat {\mathfrak g}$, we have 
$$\ \ \ \ \ \ \bar G^q = \bar G'^q\times \hat G \ \ \ \ \ \text 
{(almost direct product)}.$$

Recall that $\bar G^q/\bar K^q$ is a symmetric presentation of 
$L(q)$ and that $\bar {\mathfrak  g} '^q$ is an ideal 
of $\bar {\mathfrak g}^q$. 
Since $G'$ is transitive on the family $\{L(q): q \in M\}$ 
(see Remark \ref {effectively}), we
see that $G'$ is transitive on $M$ if and only if $\bar G'^q$ 
is transitive on $L(q)$ for some (or equivalently, for all) 
$q\in M$. 
Let $\hat {\mathfrak g }_0$ be the abelian part of 
$\hat {\mathfrak g}$ (which is regarded,
depending on the context, as an ideal of $\mathfrak g$ 
or as an ideal 
of $\bar {\mathfrak g}^q$).
Moreover, let $\check {\mathfrak {g}} '= \mathfrak {g} '\oplus 
\hat {\mathfrak g }_0$, and let $\check G '$ be the
Lie subgroup of $G$ with Lie algebra $\check {\mathfrak {g}} '$. Since 
$M$ is simply connected, we observe from Remark \ref {semisimple}  that 
$G'$ acts transitively on $M$ if and only if $\check G '$ acts 
transitively on $M$.  But this is equivalent to the fact that 
$\bar H'^q$ acts transitively on $L(q)$, where $\bar H'^q$ 
is the (normal) Lie subgroup of $\bar G^q$ which is associated 
with the ideal $\bar {\mathfrak g}'^q \oplus \hat {\mathfrak g} _0$
of $\bar {\mathfrak g}^q$. Note that this ideal contains 
always the abelian part of $\bar {\mathfrak g}^q$.
If $G'$ is not transitive on $M$, then  
$\bar H'^q$ is not transitive 
on $L(q)$. It follows from Lemma \ref {irreduciblesym} that 
the semisimple part of $\hat {\mathfrak g}$, which is a complementary 
ideal of $\bar {\mathfrak g}'^q \oplus \hat {\mathfrak g} _0$, has an 
irreducible symmetric factor $\mathfrak  g _{irr}$. 
Thus $\hat {\mathfrak g}$ has an 
irreducible symmetric factor $\mathfrak  g _{irr}$. Note that  
$\mathfrak  g _{irr}$  is an ideal of $\mathfrak g ^{\mathcal D}$ which does not depend on $q\in M$. 
Thus we have proved the following lemma:

\begin{lem}\label {ISF} If $G'$ is not transitive on $M$, then 
 $\hat {\mathfrak g}$ has a non-trivial 
irreducible symmetric factor $\hat {\mathfrak g}_{irr}$.
\end{lem}

\begin{rem}\label{exampleintersection}
Here we will present an example where $\mathfrak u$ is non-trivial. Let $M = G/H$ be
a normal homogeneous space and decompose $\mathfrak g = \mathfrak
g_\mathrm{ss} \oplus \mathfrak g_\mathrm{ab}$ as a direct sum of
ideals, where $\mathfrak g_\mathrm{ss}$ is semisimple and $\mathfrak
g_\mathrm{ab}$ is abelian. Assume that $\dim(\mathfrak
g_\mathrm{ab}) \ge 2$. Let $p = [e]$ and let $\mathbb V \subset
T_pM$ be the subspace of fixed vectors of $H$ (via the isotropy
representation). Let $\mathbb W \subset \mathbb V$ be a subspace of
codimension one. We choose $X_1, \ldots, X_{k - 1} \in \mathfrak
g_\mathrm{ab}$ such that $X_1.p, \ldots, X_{k - 1} .p$ is
a basis of $\mathbb W$. Let $\mathcal D$ be the $G$-invariant
distribution on $M$ with $\mathcal D_p = \mathbb W$. Note that
$\mathcal D$ is strongly symmetric with respect to $G$ (see
\cite{OlmosReggianiTamaru}). Let $X_k \in \mathfrak g_\mathrm{ab}$ be such that $X_k.p \in \mathbb V - \mathbb W$. Then $\mathfrak g^{\mathcal D} $ is the linear span of $X_1, \ldots, X_{k - 1}$. Moreover, if we
define $\mathfrak g' = \mathfrak g_\mathrm{ss} \oplus \mathbb R(X_{k -
1} + X_k)$, then we obtain $\mathfrak u = \bar{\mathfrak g}'^{p}
\cap \mathfrak g^{\mathcal D} = \mathbb RX_{k-1|L(p)}$.
\end{rem}

\begin{thm}\label {main} Let $M= G/H$ be an $n$-dimensional compact 
simply connected 
homogeneous Riemannian manifold, $n \geq 2$. Let $\mathcal D$
be a $G$-invariant distribution on $M$ of rank $r > 0$ which is   strongly 
symmetric  with respect to $G$, and denote by $k = n - r$  
the corank of ${\mathcal D}$. Assume 
that $M$ does not split off a symmetric factor whose associated 
parallel distribution on $M$ is contained in $\mathcal D$.
Then $k\geq 2$ and there exists a normal semisimple  Lie subgroup  
$G'$ of $G$ with $2\dim (G')\leq 
k (k+1)$ and acting transitively on $M$ such that its 
Lie algebra ${\mathfrak g}'$ is a complementary ideal of 
$\mathfrak g ^{\mathcal D} : = \{X\in \mathfrak g : 
X \text { lies on }\mathcal D\}$.  Moreover, the  equality holds if and only if 
the universal covering group 
of $G'$ is $\Spin (k+1)$.
\end{thm}

\begin {proof}
The fact that $k\geq 2$ follows from Lemma \ref {premain}, 
since $M$ is compact and simply connected. 
Let $G'$ be 
given as in Lemma \ref {premain} and
assume that $G'$ does not act transitively on $M$. 
Then, by Lemma \ref {ISF}, $\hat {\mathfrak g}$ has an 
irreducible symmetric factor  $\hat {\mathfrak g}_{irr}$
which is an ideal of  $ {\mathfrak g }^{\mathcal D}$.  
Observe that  $\mathfrak g _{irr}$  does not intersect  
$\bar {\mathfrak g} '^q$. 
Let $\tilde {\mathfrak g}$ be a complementary ideal of 
$\mathfrak g _{irr}$ in $\mathfrak g ^{\mathcal D}$. 
Let us consider the ideal 
$\tilde {\mathfrak g}'= \mathfrak g' \oplus 
\tilde {\mathfrak g}$ and its associated normal 
Lie subgroup 
$\tilde G '$ of $G$.
Then we  have the direct sum decomposition  
$\mathfrak g = \tilde {\mathfrak g}'\oplus \mathfrak g _{irr}$
of ${\mathfrak g}$ into two ideals.

Let $G_{irr}$ be the normal Lie subgroup of $G$ with Lie algebra 
$\mathfrak g _{irr}$. Then $\tilde G'$ commutes with $G_{irr}$ 
and $G = \tilde G ' \times G_{irr}$ (almost direct product). Every 
orbit  $G_{irr} \cdot q$ is a totally geodesic symmetric submanifold of 
$L(q)$ and of $M$. 
Let $K^q_{irr}$ be the isotropy group of $G_{irr}$ at $q$. Then 
$G_{irr}/K^q_{irr}$ is a symmetric presentation of $G_{irr} \cdot q$ and 
$(K^q_{irr})^o$ acts irreducibly, via the isotropy representation, 
 on $T_q(G_{irr}\cdot q)$. Since 
$(K^q_{irr})^o$ commutes with $\tilde G'$, 
it acts trivially on 
the orbit $\tilde G '\cdot q$. 
Then, since $G\cdot q = 
(\tilde G ' \times G_{irr})\cdot q = M$, 
we get
$$T_qM = T_q(\tilde G '\cdot q)  \oplus T_q(G_{irr}\cdot q)
\ \ \ \ \ \text {(orthogonal direct sum)}$$
and $\tilde G'\cdot q$ must coincide with  
the connected component of the fixed point set of $(K^q_{irr})^o$ containing
$q$. 
Then $\tilde G'\cdot q$ is a totally geodesic submanifold of $M$ and so
 the distribution $\tilde {\mathcal D}'$ on $M$, 
given by the tangent 
spaces of the $\tilde G '$-orbits, is  autoparallel. 
Moreover, this distribution  is orthogonal and complementary to the 
autoparallel distribution $\mathcal D_{irr}$, which is tangent to 
the $G_{irr}$-orbits. 
Then $\tilde {\mathcal D}'$ and $\mathcal D_{irr}$ are parallel distributions
and 
$\mathcal D_{irr}$ is contained in $ \mathcal D$. 
This contradicts the assumptions 
of this theorem and therefore $G'$ acts transitively on $M$. 
 The other statements follow from 
Lemma~\ref {premain}.
\end {proof}

\begin{rem}\label {cohomogeneity1} We recall here a 
well-known fact. 
Let $M$ be a complete and 
simply connected Riemannian manifold. 
Let $H$ be a connected Lie subgroup of $I (M)$ which admits 
a bi-invariant Riemannian metric. Assume that 
all $H$-orbits have   codimension one in $M$, that is, $H$ acts with
cohomogeneity 
one on $M$ and there are no singular orbits. 
Then $M$ splits as $M = N \times \mathbb R$
(generally not a Riemannian product). 
For the sake of 
completeness we will sketch the proof. 

 Let us change the 
Riemannian metric $(\ ,\ )$ 
on $M$  along the distribution $\mathcal T$ given by the tangent spaces of
 the $H$-orbits. The new metric at $q \in M$, restricted 
to $\mathcal T_q$, is the normal homogeneous metric on the orbit
$H\cdot q$ at $q$ (this is a local construction 
and it does not depend on 
whether  the orbit is exceptional or not). 
The group $H$ acts also by isometries on $M$ with this new Riemannian metric. 
If $\gamma (t)$ is a geodesic which is perpendicular at 
$\gamma (0)= p$ to the orbit $H\cdot p$, then it is always perpendicular 
to the $H$-orbits (since a Killing field projects constantly on any 
geodesic). So the 
distribution $\nu$ perpendicular to the 
$H$-orbits is an autoparallel distribution of rank one. Moreover, 
the one-parameter perpendicular variation of orbits $H\cdot \gamma (t)$ 
(we consider these orbits only locally around $\gamma (t)$) 
is by isometries. Then the $H$-orbits are totally geodesic 
and hence $\mathcal T = \nu ^\perp$ is also an autoparallel 
distribution.  It follows that $\nu$  is  a parallel 
distribution and then, by the de Rham decomposition theorem, $M$ splits 
off a line. 
\end{rem}

\begin{rem} \label {semisimple} 
Let $M$ be a compact and simply connected Riemannian 
homogeneous space. Let $G$ be a Lie subgroup of $I (M)$ 
which acts transitively on $M$.
Then the semisimple part $G_{ss}$ 
of $G$ acts also transitively 
on $M$. In fact, 
let $\mathfrak g = \mathfrak g_{ss} \times \mathbb R ^k$,
where $\mathfrak  g_{ss}$ is semisimple. We  always have 
 such a decomposition 
since $I (M)$ is compact and therefore $G$  admits a bi-invariant metric.
Let $0\leq d\leq k$ be the smallest integer such that 
 the Lie subgroup of $G$ with Lie algebra 
$\mathfrak g_{ss} \times \mathbb R ^d$ is transitive on $M$. 
If $d\geq 1$, let $\bar G$ be the Lie subgroup 
of $G$ with Lie algebra 
$\mathfrak g_{ss} \times \mathbb R ^{d-1}$. 
Then all orbits of $\bar G$ have codimension one in $M$. 
This is a contradiction since $M$ is compact and simply connected 
(see Remark \ref {cohomogeneity1}). Therefore we must have $d=0$. 
\end{rem}

We will need the  following result from \cite{OlmosReggiani1} for the proof of 
next lemma which we will use later. 

\begin{prop}[see Lemma 5.1 in \cite{OlmosReggiani1}]\label{normal}
  Let $M = G/H$ be a homogeneous Riemannian manifold (where $G$ is not
  necessarily connected), $p = [e]$ and $\Phi$ be a normal
  subgroup (eventually, finite) of the isotropy group $H$ at $p$. Let
  $\mathcal D^\Phi$ be the $G$-invariant distribution on $M$ such that
  $\mathcal D^\Phi_{g(p)} \subset T_{g(p)}M$ is the subspace
  of fixed vectors of $g \Phi g^{-1}$. Then $\mathcal D^\Phi$ is an
  autoparallel distribution.
\end{prop}

\begin{lem}\label{A}
  Let $M = G/H$ be a compact homogeneous Riemannian manifold and 
  $\mathcal D^1$ be an autoparallel $G$-invariant distribution on $M$
  which is strongly symmetric with respect to $G$. Let $\mathcal D^2$
  be an autoparallel $G$-invariant distribution on $M$ such that
  $\mathcal D^1 \subset  \mathcal D^2$ and ${\rm rk}(\mathcal D^2) =
  {\rm rk}(\mathcal D^1) +1 $. Then $\mathcal D^2$ is strongly symmetric
  with respect to $G$.
\end{lem}

\begin {proof}
  Let $q \in M$ and $L^i(q)$ be the maximal integral manifold of
  $\mathcal D^i$ containing $q$, $i = 1, 2$. Let $v \in T_q(L^1(q))$ be
  arbitrary. Then, since $\mathcal D^1$ is strongly symmetric, there
  exists $X \in \mathfrak g$, regarded as a Killing field, such that
  $X. q = v$ and $\langle\nabla_wX, z\rangle = 0$ for all $w, z
  \in T_q(L^1(q))$. Since $\mathcal D^i$ is $G$-invariant,
  $X_{\vert L^i(q)}$ must always be tangent to $L^i(q)$, $i = 1, 2$. 
  
  Let $\xi \in \mathfrak g$ be such that $0 \neq \xi.q \in
  \mathcal D^2_q$ and $\xi.q$ is orthogonal to $\mathcal
  D^1_q$. Since the projection of $\xi_{\vert L^1(q)}$ to the tangent
  space of $L^1(q)$ is a bounded Killing field, it lies in the Lie
  algebra of the transvection group of $L^1(q)$. Since $\mathcal D^1$
  is strongly symmetric, there exists $Y \in \mathfrak g$ such that
  $Y_{\vert L^1(q)}$ is always tangent to $L^1(q)$ and coincides with
  the projection of $\xi_{\vert L^1(q)}$ to the tangent spaces of
  $L^1(q)$. So, by replacing $\xi $ by $\xi - Y$, we may assume that
  $\xi_{\vert L^1(q)}$ is always perpendicular to $L^1(q)$. Note that 
  $\xi_{\vert L^2(q)}$ must always be tangent to $L^2(q)$. 
  
  If $\eta \in \mathfrak g$ is tangent to $L^2(q)$ and perpendicular
  to $L^1(q)$, then $\eta$ must be a scalar multiple of $\xi$. In
  fact, let $\lambda \in \mathbb R$ such that $\lambda(\xi.q) =
  \eta.q$. Then $\psi = \eta - \lambda\xi$ vanishes at $q$ and so
  $\psi.q \in T_q(L^1(q))$. Since $L^1(q)$ is $G$-invariant,
  $\psi$ must always be tangent to $L^1(q)$. However, $\psi $ is always
  perpendicular to $L^1(q)$, and therefore $\psi$ is identically zero on
  $L^1(q)$. Since the totally geodesic submanifold $L^1(q)$ of
  $L^2(q)$ has codimension one, we get $\eta _{\vert L^2(q)} = 0$. We
  may have chosen, by making use of a bi-invariant metric on
  $\mathfrak g$, $\xi \in (\mathfrak g_0)^\perp$, where $\mathfrak g_0
  = \{X \in \mathfrak g: X_{\vert L^2(q)} = 0\}$. Let $G^1$ be the connected
  component of the subgroup of $G$ that leaves $L^1(q)$ invariant. If
  $g \in G^1$, then $g_*\xi = \Ad(g)\xi \in (\mathfrak g_0)^\perp$ is
  tangent to $L^2(q)$ and perpendicular to $L^1(q)$. Then $\Ad(g)\xi$
  is a scalar multiple of $\xi$. Since $\Ad(g): (\mathfrak g_0)^\perp
  \to (\mathfrak g_0)^\perp$ is an isometry and $G^1$ is connected,
  we get $\Ad(g)\xi = \xi$ and so $\xi$ commutes with $\mathfrak g^1$.

  Now observe that for all $z \in T_q(L^1(q))$ we have
  $\langle\nabla_{\xi.q}X, z\rangle = -\langle\nabla_zX, \xi.q\rangle = 0$, 
  since $X$ is tangent to the totally geodesic submanifold $L^1(q)$ of
  $M$. As
  $\langle\nabla_{\xi \cdot q}X, \xi \cdot q\rangle = 0$,
  we conclude that $X_{\vert L^ 2(q)}$ is a transvection at $q$. 
  
  Let us now prove that $\xi_{\vert L^2(q)}$ is also a
  transvection at $q$. Let $X$ be as above. Since $[X, \xi] = 0$ we
  obtain $\nabla_{X. q}\xi = \nabla_{\xi.q}X =
  0$. Observe also that $\langle\nabla_{\xi .q}\xi, v\rangle =
  -\langle\nabla_v\xi, \xi.q\rangle = 0$, where $v \in
  T_q(L^1(q))$ is arbitrary. Since $\langle\nabla_{\xi.q}\xi,
  \xi.q\rangle = 0$, we conclude that $\xi_{\vert L^1(q)}$ is a
  transvection at $q$. It follows that $\mathcal D^2 $ is strongly symmetric.
  \end {proof}
  
  The proof was rather involved, since we had to use that $\mathfrak
  g$ admits a bi-invariant metric. Otherwise, if we consider for example the
  hyperbolic plane $H^2$ as a solvable Lie group
  $S$, the distribution tangent to the lines that meet at infinity is
  $S$-strongly symmetric, but the distribution $TH^2$ is not
  $S$-strongly symmetric.

\section {The index of symmetry}

In this section we present the definition and some basic facts about the index
of symmetry, for details we refer to \cite {OlmosReggianiTamaru}.
Let $(M, \langle \cdot , \cdot \rangle)$ 
be an $n$-dimensional Riemannian manifold with Riemannian metric  
$\langle \cdot , \cdot \rangle$. 
We denote  by $\mathfrak{K} (M)$ the Lie algebra
of global Killing fields on $M$.
The {\it Cartan subspace} $\mathfrak p ^q$ at $q\in M$ is 
 $$\mathfrak p ^q : = \{ X \in  \mathfrak {K} (M): \,
(\nabla X)_q = 0\},$$
where $\nabla$ is the Levi-Civita connection of $M$.
The elements of $\mathfrak p ^q$ are called {\it 
transvections at } $q$.
The {\it symmetric isotropy subalgebra} at $q$ is 
$$\mathfrak k ^q : = \text {linear span of  } \{ [X,Y]: \, X,Y \in
\mathfrak p ^q\}.$$
For $X,Y \in \mathfrak p ^q$ we have $[X,Y]_q =
(\nabla _X Y)_q - (\nabla _Y X)_q = 0$.
Thus  $\mathfrak k ^q$
 is contained in the full isotropy algebra
$\mathfrak {K}_q (M) = \{X \in {\mathfrak K}(M) : X_q = 0\}$. 
Moreover, since $\mathfrak p ^q $ is invariant under the action of the isotropy algebra at $q$,
$$ \mathfrak g ^q : = \mathfrak k ^q \oplus \mathfrak p ^q$$
is an involutive Lie algebra.
Let $G^q$ and $K^q$ be the Lie subgroup of 
$I(M)$ with Lie algebra
 $\mathfrak g ^q$ and $\mathfrak k ^q$, respectively.

 The {\it symmetric subspace} $\mathfrak s _q$ of $T_qM$ at $q \in M$ is defined
by
$$ \mathfrak s _q: =\{X_q : X \in \mathfrak p ^q\}  . $$
The {\it index of symmetry} $i_{\mathfrak s}(M)$ of $M$ 
  is the infimum of $\{\dim (\mathfrak s _q) : q \in M\}$.
Note that $\dim (\mathfrak s _q) = 
\dim (\mathfrak p ^q) = \dim (L(q))$, 
where $L(q):= G^q \cdot q$ is the so-called {\it leaf 
of symmetry} containing $q$. 
The {\it coindex of symmetry} $ci_{\mathfrak s}(M)$ of $M$ is  defined by
$ci_{\mathfrak s}(M) = n-i_{\mathfrak s}(M)$.

\begin {facts} [see \cite {OlmosReggianiTamaru}, Section 3] \label {facts} 
Let $q \in M$.
\begin{enumerate}[\upshape (a)]
\item $G^{h(q)} = hG^qh^{-1}$
and $d_qh(\mathfrak s_q) = \mathfrak s_{h(q)}$ 
 for all $h\in I(M)$.
\item $L(q)$ 
is a totally geodesic submanifold 
of $M$ and a globally symmetric space. 
\item $G^q$ is a normal subgroup of 
$\{g\in \text {\rm I} (M): g(L(q)) = L(q)\}$ and 
$K^q$ is a normal subgroup of the full isotropy group
$I(M)_q$.
\item If $X\in \mathfrak p ^q$, then $\gamma (t) = 
\Exp(tX)(q)$ is a geodesic in $M$. Moreover, the parallel transport
along $\gamma$ from $q = \gamma(0)$ to $\gamma(t)$ is given by $d_q\Exp(tX)$.
\item For every $I ^o (M)$-invariant 
tensor field $T$ on $M$ we have $\nabla _{X_q}T = 0$ for all $X\in \mathfrak
p^q$. 
In particular, $\nabla _{X_q}R = 0$, where $R$ is the Riemannian curvature
tensor of $M$. 
\item If $X\in \mathfrak p ^q$ and $Z$ is 
any vector field on $M$, then 
$\nabla _{X_q}Z = [X,Z]_q$.
\item If $M$ is compact, then $G^q$ acts almost effectively on $L(q)$. 
\end{enumerate}
\end {facts}

In this paper we will only deal with compact homogeneous Riemannian manifolds
$M=G/H$. 
In this case $q\mapsto \mathfrak s_q$ is a $G$-invariant, and hence smooth,
distribution which is called the 
{\it distribution
of symmetry} of $M$. The distribution $\mathfrak s$ on $M$ is autoparallel 
and the leaves of symmetry $L(q)$ are the maximal integral manifolds of
${\mathfrak s}$. 
Note that the distribution of symmetry is a strongly 
symmetric distribution with respect to $I^o(M)$.
Let $\mathcal K (M)^{\mathfrak s}$ be the ideal of 
$\mathcal K (M)$ which consists of those Killing fields 
that are tangent to $\mathfrak s$. 

\begin{rem} 
$G^q$ is a Lie subgroup of $ I(M)$ but it is not necessarily contained in
the presentation group $G$ of $M$. 
In the notation of Section 3, if $\mathcal D = \mathfrak s$
and $G= I^o (M)$, then $\bar G ^q = G^q_{\vert L(q)}$. 
\end{rem}

\section {Structure results for spaces with non-trivial 
index of symmetry}

In this section we develop some general structure theory in relation to the index and co-index of symmetry.
 These results are useful for understanding the 
geometry of (irreducible) compact homogeneous 
spaces with a non-trivial index of symmetry.
Our main theorem is crucial for classifying 
compact homogeneous 
spaces $M^n$ with low co-index of symmetry $k = ci_{\mathfrak s}(M)$, since it gives a bound 
on the dimension of a transitive group, and hence on $n$, in terms of 
$k$. 

\begin{rem}\label {Jacobi} (The Jacobi operator in directions
of the distribution of symmetry). If $X\in \mathfrak p ^q$ then, 
from (d) and (e) of Facts \ref {facts}, 
$\nabla_{\gamma ' (t)}R = 0$, where $\gamma (t) = 
\text {Exp} (tX)(q)$ is the  geodesic with initial 
condition $\gamma ' (0) = X_q$. 
 Let  $e_1 = X_q ,e_2, \ldots , e_{n}$ be an orthonormal basis 
 of $T_qM$ which 
diagonalizes the Jacobi operator $R_{\cdot, X_q}X_q$ at $q$
 with corresponding eigenvalues $a_1 = 0, a_2,\ldots,  a_n$.    
Then $e_1(t),\ldots, e_n(t)$ diagonalizes 
$R_{\cdot , \gamma ' (t)}\gamma ' (t)$ with the same 
corresponding eigenvalues, where $e_i(t)$ denotes 
the parallel transport of $e_i$ along $\gamma (t)$.
For $\kappa \in {\mathbb R}$ we define
\[ 
\sin_\kappa(t) = 
\begin{cases}
\frac{1}{\sqrt{\kappa}}\sin(\sqrt{\kappa}t) & ,\ {\rm if}\ \kappa > 0, \\
t & ,\ {\rm if}\ \kappa = 0,\\
\frac{1}{\sqrt{-\kappa}}\sinh(\sqrt{-\kappa}t) & ,\ {\rm if}\ \kappa < 0,
\end{cases}
\]
and
\[ 
\cos_\kappa(t) = 
\begin{cases}
\cos(\sqrt{\kappa}t) & ,\ {\rm if}\ \kappa > 0, \\
1 & ,\ {\rm if}\ \kappa = 0,\\
\cosh(\sqrt{-\kappa}t) & ,\ {\rm if}\ \kappa < 0.
\end{cases}
\]
Let $v= v_1 e_1 + \ldots + v_n e_n$ and 
$w = w_1 e_1 + \ldots + w_n e_n$. Then the Jacobi field 
$J(t)$ along $\gamma (t)$ with initial conditions 
$J(0) = v$ and $J'(0)= w$ is given by 
\[
J(t) = \sum_{i=1}^n v_i\cos_{a_i}(t)e_i(t) + \sum_{i=1}^n
w_i\sin_{a_i}(t)e_i(t).
\]
Let now $Y \in \mathfrak {K} (M)$ be a Killing field 
with $Y_q = e_i$. 
Then $J_Y(t) = Y _{\gamma (t)}$ is a Jacobi 
field along $\gamma (t)$ with $J_Y(0) = e_i$. 
Since $M$ is compact, $Y(t)$ is bounded and thus also $J_Y(t)$ is 
bounded for $t\in \mathbb R$. 
From the above description of the Jacobi fields along $\gamma$ it follows that
$a_i \geq 0$ for all $i = 1,\ldots,n$.
Therefore the Jacobi operator $R_{\cdot , X_q}X_q$ is positive 
semidefinite. 
\end{rem}

\begin{prop}\label {structure}
Let $M$ be a homogeneous compact Riemannian manifold with 
a non-trivial index of symmetry.
Let  $I^{q}(M)$  be the Lie subgroup of  
$I (M)$ that
leaves invariant the leaf of symmetry $L(q)$. 
We identify $\mathfrak K (M)$ with the Lie algebra of $I(M)$ and define
 $$\mathfrak m ^q = \{\xi \in \mathfrak K (M): 
\xi _{\vert L(q)} \text { is always perpendicular to } L(q)\}.$$ 
Then the following statements hold:
\begin{enumerate}[(i)]
\item $\mathfrak m^q$ is an 
$\Ad (I^{q}(M))$-invariant
subspace of $\mathfrak K (M)$.
\item The linear map  $ \mathrm {Ev}^q:
\mathfrak m^q \to (T_qL(q))^\perp,\ \xi \mapsto \xi_q$ is surjective and
$$\operatorname  {ker}(\mathrm {Ev}^q) 
= \{
\xi \in \mathfrak K (M): \xi _{\vert L(q)}= 0\}.$$
\item Let $0\neq X\in \mathfrak p ^q$ 
be a transvection at $q$ and let $\gamma (t) = 
\Exp(tX)(q)$. 
Decompose $T_{\gamma (t)} M = E_0(t)\oplus \ldots \oplus E_r(t)$ 
 ($E_0$ may be trivial)
into the eigenspaces associated to 
the different (constant) eigenvalues $0 = \lambda _0< \ldots
<\lambda _r$ of the Jacobi operator 
$R_{\cdot,\gamma ' (t)}\gamma ' (t)$. 
Let $\xi \in \mathfrak K (M)$ and let $(\xi _{\gamma (t)} )^i$ 
be the orthogonal projection 
of $\xi _{\gamma (t)}$ onto $E_i(t)$. Then there exists $\eta \in \mathfrak K
(M)$
such that $\eta _{\gamma (t)} = (\xi _{\gamma (t)} )^i$.
\end{enumerate}
\end{prop}

\begin {proof}
(i) For every $g \in I(M)$ the adjoint transformation $\Ad(g)$ maps Killing
fields to Killing fields. 
If, moreover, $g\in I^q(M)$, then $g(L(q)) = L(q)$, and thus $\Ad(g)$ maps
any Killing field which 
is perpendicular to $L(q)$ into a Killing field which 
is perpendicular to $L(q)$.
This proves the statement in (i). 

(ii) Let $w\in (T_qL(q))^\perp$ and choose  
$Z\in \mathfrak K (M)$ with $Z_q = w$. 
The orthogonal projection  $\bar Z^T$ of  $Z_{\vert L(q)}$ to $TL(q)$ is an 
intrinsic transvection of $L(q)$ since $\bar Z^T$ is bounded. Thus there 
exists $Y\in \mathfrak g ^q$ such that $ Y_{\vert L(q)} = \bar{Z}^T$. 
Then $Z-Y$ is always perpendicular to $L(q)$ and $\mathrm {Ev}^q(Z-Y)= 
(Z-Y)_q = w$. This shows that $\mathrm {Ev}^q$ is surjective. Let $\xi \in
\mathfrak m^q$ with 
$\xi _q = 0$. Then $\xi _q \in T_qL(q)$. Hence, 
since the foliation of 
symmetry $\mathcal L = 
\{L(q): q \in M\}$ is invariant under isometries, $\xi_{\vert L(q)}$ 
must always be tangent to 
$L(q)$. Therefore $\xi _{\vert L(q)} = 0$, which implies the second statement in
(ii).

(iii)  Since $X \in \mathfrak p ^{\gamma (t)}$, we have
  $\nabla _{X_{\gamma (t)}}\xi = 
\nabla _{X_{\gamma (t)}}\xi -
\nabla _{\xi _{\gamma (t)}}X = [X,\xi ]_{\gamma (t)}$, and therefore
$$[X,[X, \xi]]_{\gamma (t)} = \frac {\,D^2}{dt^2}
(\xi _{\gamma (t)})= 
-R_{\xi _{\gamma (t)},\gamma ' (t)}\gamma ' (t)$$
by the Jacobi equation. 
Let $J_i(t)$ be  the orthogonal projection onto $E_i(t)$ of the Jacobi 
field $J^\xi (t) ) =\xi_{\gamma (t)}$, $i= 0, \ldots , r$. Observe 
that $J_i(t)$ is a Jacobi field. 
Let $L : \mathfrak K (M) \to \mathfrak K (M)$ be the linear 
map defined by $L(\eta )= [X,[X,\eta]]$. 
Then $$L(\xi)_{\gamma (t)} = \lambda_0 J_0(t) + \ldots +\lambda _rJ_r(t), $$
where $-\lambda _i \geq 0$ is the eigenvalue 
of the Jacobi operator $R_{\cdot,\gamma ' (0)}\gamma ' (0)$
associated to
$E_i(0)$ ($\lambda _0 = 0$). Let us write 
$$L^j (\xi)_{\gamma (t)}  
= \lambda_0 ^jJ_0(t) + \ldots +\lambda _r^jJ_r(t)$$ 
for $j= 0, \ldots , r-1$, where $L^0 (\xi) = \xi$. 
The vectors  $v_0, \ldots, v_{r}$ 
of $\mathbb R ^{r+1}$ are linearly independent, where 
$v_j = (\lambda _j^0,\lambda _j^1 , \ldots ,\lambda _j^{r})$, 
$j=0, ..., r$   
(since the  determinant of Vandermonde is not zero). 
It is not hard to see that for every $i \in \{0,\ldots,r\}$ 
 there exist scalars $c(i)_0, \ldots, c(i)_{r}$
such that $$c(i)_0 \xi _{ \gamma (t)} + 
c(i)_1 L^1(\xi)_{ \gamma (t)} + \ldots + 
c(i)_r L^r(\xi)_{ \gamma (t)} = L^i ( \xi)_{\gamma (t)} = J_i(t). $$
Then $\eta = L^i ( \xi)$ has the desired properties. 
\end {proof}

We have the following stronger version of Theorem \ref {main}
for  the distribution of symmetry, which is a consequence of Theorem \ref
{main}, except for 
the last assertion which follows from Lemma \ref {isotropy}.

\begin{thm}\label {mainco}
Let $M$ be a compact, simply connected, Riemannian homogeneous manifold
with coindex of symmetry $k$. 
Assume that $M$ does not split off a symmetric de Rham factor. 
Then $k\geq 2$ and there exists a transitive semisimple normal Lie subgroup 
$G'$ of 
$I(M)$, whose Lie algebra is a complementary ideal to 
 $\mathfrak K (M)^{\mathfrak s}$,   
 such that $2\dim (G')\leq k(k+1)$. The
 equality holds if and only if 
the universal covering group 
of $G'$ is $\Spin (k+1)$. Moreover, if the equality holds and  
$k\geq 3$, then the isotropy group of $G'$ has positive dimension. 
\end{thm}

\begin{lem}\label {isotropy}
Assume that in Theorem \ref {mainco} the equality holds and so
 $G' = \Spin (k+1)$ acts transitively by isometries on $M$ 
(almost effective action). 
Then, if $k\geq3$, the isotropy 
group $\Spin (k+1)_q$  at $q \in M$ has  positive dimension
(or equivalently, since $M$ is simply connected, $\Spin (k+1)_q$ is not
trivial).
\end{lem}

\begin {proof}
Assume that the isotropy group $\Spin (k+1)_q$ is trivial.
Let $\mathfrak s$ be the distribution of symmetry, which has dimension  
$\frac 1 2 k(k-1)$, since $\dim (\Spin (k+1)) = \frac 1 2 k(k+1)$.  
Let $q\in M$ and  define
$$\Spin (k +1)^q = \{g\in \Spin (k +1): g(L(q))=L(q)\}.$$
Since the isotropy group $\Spin (k +1)_q $ is trivial, the group 
$\Spin (k +1)^q$ acts effectively on $L(q)$ and so it can 
be identified with the group 
$$\overline {\Spin} (k +1)^q = \{g_{\vert L(q)}
\in \Spin (k +1): g(L(q))=L(q)\}_o.$$
From Theorem \ref {mainco} the isometry algebra 
is given by the following sum of ideals:
\begin {equation}
\label {eq21-10}
 \mathfrak K (M) =
\mathfrak {so}(k+1) \oplus  \mathfrak K (M)^{\mathfrak s}.
\end {equation}

In the notation of this section, since $\Spin (k+1)$ is a normal 
subgroup of $I(M)$, $\overline {\Spin }(k+1)^q$ 
is a normal subgroup of   $\bar G^q$, where 
$\bar G^q$ is  the transvection group at $q$, restricted to $L(q)$.
Then, since $\overline {\Spin }(k+1)^q$ acts simply transitively on 
$L(q)$, $L(q)$ must be a Lie group with a bi-invariant Riemannian
metric (see Lemma \ref {irreduciblesym}).
In general, $L(q)$ could be non-simply connected.  
Observe that no element $g\in I (M)^{\mathfrak s}$, the subgroup of $I(M)$
associated with the ideal  $\mathfrak K(M)^{\mathfrak s}$, can  belong to   
the full isotropy group $I (M)_q$. In fact, since $g$ commutes with 
 $ \Spin (k+1)^q$, which is transitive on $L(q)$, 
 $g$ must be the identity on $L(q)$ and therefore $g = e$ (see Remark  
\ref {effectively}). Note also that $ {\Spin }(k+1)^q$ is semisimple, 
since the quotient $\Spin (k+1)/ {\Spin }(k+1)^q$ is (equivariantly) 
isomorphic 
to $\SO (k+1)/\SO (k)$ (see the proof of Lemma \ref {premain}).
Then $L(q)$ has no flat  factor locally. 
Using (\ref {eq21-10}) this implies  that
 $\dim (\mathfrak K(M)^{\mathfrak s}) = 
\dim (L(q)) = \dim  (\Spin (k+1)^q)$ and that 
$\mathfrak g ^q = \mathfrak {so} (k) \oplus 
\mathfrak K (M)^{\mathfrak s} \simeq \mathfrak {so} (k) 
\oplus \mathfrak {so} (k) $. 
 
Then  $I^o(M) = \Spin (k+1)\times \Spin ' (k)$, 
where the 
second factor is the subgroup 
$\Spin (k)\subset \Spin (k+1)$, but acting  from the
right on $M \simeq \Spin (k+1)$,
that is, if $g\in \Spin ' (k)$ then $g(q) = qg^{-1}$.
Note that $I (M)^{\mathfrak s}$ must be transitive 
on $L(q)$ and so on any maximal integral manifold of $\mathfrak s$. 
This implies that the Riemannian metric on $M = \Spin (k+1)$ induces a
Riemannian 
submersion onto the quotient
 \[ \Spin (k+1)/\Spin (k+1)^q
\simeq \SO (k+1)/\SO (k),\]
 which is a sphere. 
We are now in the following situation: 
\begin{enumerate}[(a)]
 \item $M = \Spin (k+1)$.
\item $I^o (M) = \Spin (k+1)\times \Spin '(k)$.
\item The distribution of symmetry is 
$$g\mapsto \so '(k) g = \Ad (g)(\so (k)g), \ 
g\in M\simeq \Spin (k+1).$$ 
\item The maximal integral manifolds of the distribution of symmetry are 
\[L(g) = \Spin '(k)g = g\Spin (k).\]
\item The isotropy group at $e$ is  $$(I ^o (M))_e 
= \text {diag}(\Spin (k)) = \{(h,h)\in \Spin (k)\times \Spin '(k)
: h\in \Spin (k)\}.$$
\item $ K^e = (I ^o (M))_e $, $\mathfrak k ^e = 
\text {diag}(\so (k))$, $\mathfrak p ^e = \{(v,-v) \in 
\so (k) \times \so '(k)\}$,   $G^e =  
\Spin (k)\times \Spin ' (k)$, $\mathfrak g ^e = 
\so (k) \oplus \so '(k)$.
Recall that $K^e$ acts 
almost effectively on $L(e)$ (see Facts \ref {facts})).
\end{enumerate}

Let $X\in \so (k+1)\subset \so (k+1) \oplus \so '(k)\simeq 
\mathfrak K (M)$. Then the orthogonal projection $\bar X$ of $X_{\vert L(e)}$ to
$TL(e)$ is a bounded Killing field on $L(e)$ and so it belongs to 
$\mathfrak g ^e_{\vert L(e)}$.
Since $X$ commutes with any Killing 
field induced by $\so '(k)$, and $\Spin ' (k)$ preserve the distribution 
of symmetry, we see that $\so ' (k)_{\vert L(e)}$ commutes with $\bar X$.
Then there must exist 
$Z\in \so (k)$ such that 
 $\bar X  = \bar Z$, where $\bar Z$ denotes the restriction of $Z$ to 
$L(e)$. Then $ Y = X-Z\in \so (k+1)$ is a Killing field whose restriction to
$L(e)$ is 
always perpendicular to $L(e)$.   Note that in this way we can construct such a
Killing 
field $Y$ with an arbitrary initial condition $Y_e \in \mathfrak s ^
\perp$. 

Let $$\mathfrak m = \{ Y \in \so (k+1): Y_{\vert L(e)} 
\text { is perpendicular to }L(e)\}.$$
Then $\mathfrak m$ is an $\Ad (\Spin (k))$-invariant
complementary subspace of $\so (k) $ in $\so (k+1)$.
By Lemma \ref {uniquereductive}, if $k\neq 3$, 
 $\mathfrak m = 
\so (k) ^\perp$, where the orthogonal complement is with 
respect to the Killing form of $\so (k+1)$. 
We equip  $M\simeq \Spin (k+1)$ with the bi-invariant Riemannian metric 
$(\cdot , \cdot)$. Note  that  
$I^o (M) = \Spin (k+1) \times \Spin  (k) 
\subset I^o (M,(\cdot , \cdot))
= \Spin (k+1)\times \Spin ' (k+1)$. 

If $\xi , \eta \in \mathfrak m = 
\so (k) ^\perp$, then these two Killing fields are perpendicular 
to $L(e) =\Spin (k)\cdot e$ with respect to both Riemannian metrics $(\cdot ,
\cdot )$ 
and $\langle \cdot , \cdot\rangle$ (the given one). 
Moreover,  if $X\in \mathfrak p ^e$, then $X$ is a parallel vector field 
at $e$ with respect to both metrics. 
Note that the canonical projection to $S^k = \Spin (k+1)/\Spin (k)$ 
is a Riemannian submersion 
(up to rescaling) with respect to any of the two metrics on $M$. 
So, up to rescaling, $(\cdot , \cdot )$ coincides with 
$\langle \cdot , \cdot\rangle$ on $\so (k) ^\perp \simeq (\mathfrak s _e)^
\perp$. 
Unless $(\cdot , \cdot ) = \langle \cdot , \cdot \rangle$, this contradicts
the 
so-called bracket formula of Proposition 3.6 of 
\cite {OlmosReggianiTamaru}:
\begin{equation} \label{star1}
2\langle [\xi , X], \eta \rangle _e = - \langle X, [\xi , \eta ]\rangle _e\ ,\ 
2(   [\xi , X], \eta) _e = -  ( X, [\xi , \eta ])_e ,
\end{equation} 
taking into account  that $[\so (k)^\perp , \so (k)^\perp] = 
\so (k)$. 
Then, if $k\neq 3$, $M\simeq \Spin (k +1)$ has a 
bi-invariant metric and thus $M$ is a symmetric space, which is
a contradiction, since the coindex of symmetry is $k$. 
Therefore the isotropy group is non-trivial if $k\neq 3$. 

The case $k=3$ is more involved since $\SO (4)$ 
is not simple. 
Since $\Spin (4)$ acts almost effectively on the quotient 
$\Spin (4)/ \Spin (4)^e$ of $M$ by the leaves of symmetry
(see the proof of Lemma \ref {premain}),
we see that $\Spin (4) ^ e$ cannot be a factor of $\Spin (4)$. Then, according
to Remark \ref {quotient}, $\Spin (4) ^ e \simeq \Spin (3)$ 
is the  subgroup of $\Spin (4)$ which is equivalent 
to the diagonal inclusion of $\Spin (3)$ in $\Spin (4) 
= \Spin (3) \times \Spin (3)$. 
As remarked above, 
$\mathfrak m = \{ Y \in \so (4): Y_{\vert L(e)} 
\text { is perpendicular to }L(e)\}$
 is an $\Ad (\Spin (3))$-invariant
complementary subspace of $\so (3) $ in $\so (4)$ 
and gives a reductive decomposition of 
$\Spin (3)\times \Spin (3)/ \text {diag} (\Spin (3))$. 

We still have to deal with the cases (1) and (2) of Remark 
\ref {uniquereductive}.
In the first case $\mathfrak m$ is the orthogonal complement 
with respect to an $\Ad (\SO (4))$-invariant bilinear form
$Q$. 
Such a form $Q$ is equal to $B$ on the first ideal  of $\so (4) = \so (3) 
\oplus \so (3)$ and equal to $\lambda B$ on the second ideal, 
where $0\neq \lambda \neq -1$ and $-B$ is the Killing form 
of $\so (3)$.
The bilinear form $Q$ induces on  $M= \Spin (4)$  a 
bi-invariant pseudo-Riemannian metric. Then $M$ is a pseudo-Riemannian
product of $\Spin (3)$ with a bi-invariant Riemannian metric and 
$\Spin (3)$ with a Riemannian or anti-Riemannian metric 
(depending on the sign of $\lambda$). 
If $(\cdot , \cdot ) = Q$ we get 
the same contradiction as in (\ref{star1}) unless $\langle \cdot , \cdot
\rangle$ 
is proportional to $Q$. Thus $Q$ is positive definite 
and $M$ is a symmetric space. which gives a contradiction. Therefore 
the isotropy group cannot be trivial. 

Let us now consider case (2) of Remark 
\ref {uniquereductive}, where 
$\mathfrak m \simeq (\so (3), 0) \subset \so (3) \oplus \so (3)$ 
(the other case $\mathfrak m \simeq (0, \so (3))$  is analogous). 
In this case, the distribution perpendicular to $\mathfrak s$ 
is integrable with maximal integral manifolds  $H \cdot q$, where 
$H$ is the first factor of $\Spin (4)$. 
Since the projection of $M$ onto the quotient of $M$ by the leaves
of symmetry is a Riemannian submersion, 
the orbit $H  \cdot q$ is a totally geodesic submanifold of $M$ for every $q\in
M$. 
Thus $(\mathfrak s)^\perp  $ and $\mathfrak s$ 
 are autoparallel distributions and hence both parallel distributions.
This implies that $M$ is a Riemannian product, which is a contradiction. 
Altogether we conclude now that the isotropy group of $\Spin (4)$ is not
trivial.
\end {proof}

\begin{rem} \label {uniquereductive} 
The second and third author observed in Remark 2.1 of \cite {OlmosReggiani2}
that there is only one naturally reductive decomposition on the homogeneous
space
 $\SO (n+1)/\SO (n)$ if $n\neq 3$. 
The assumption that the reductive 
decomposition is naturally reductive is not necessary. In fact, let
$\nabla $ be the Levi-Civita connection on  
$S^n = \SO (n+1)/\SO (n)$ and
$\nabla ^c$ be the canonical
connection associated with a reductive decomposition on the homogeneous space
$\SO (n+1)/\SO (n)$, and define $D = \nabla -\nabla ^c$.
We will show that $D$ is totally skew. 
Since $\nabla ^c $ is a metric connection, we have
$\langle D_XY,Y\rangle =0$ for all vector fields $X,Y$ on $S^n$. So we only need
to show 
that $\langle D_XX,Z\rangle =0$ for perpendicular vector fields $X,Z$ on $S^n$.
Since for $n= 1$ there is no isotropy group, we have $D =0$. 
If $n=2$ then there is only one reductive decomposition 
$\so (3) = \so (2) + \mathbb V$, where $\mathbb V$ is the orthogonal 
complement to $\so (2)$ with respect to the Killing form of $\so (3)$.
This is because of the fact that $\mathbb V$ is the only irreducible 
$\SO (2)$ invariant subspace. 

Thus we may assume that 
$n\geq 3$. Let $h\in \SO (n+1) _q \simeq \SO (n) $ 
be such that $h(q) = q$, $dh(x) = x$ and 
$dh(z) = -z$. 
Then, since $D$ is $\SO (n+1)$ invariant $\langle D_xx,z\rangle
= 0$. Then $D$ is totally skew  and $\nabla ^c$ is 
associated with a naturally reductive decomposition. 
Moreover, $D$ is parallel (since it is invariant under the transvections 
of $S^n$). Hence $\langle D_..,.\rangle$ is a harmonic $3$-form which represent 
a $3$-cohomology class on $S^n$. Then $D=0$, if $n\neq 3$. 

Observe that for $n=3$ the above argument implies that 
$D$ is also totally skew. So  a reductive decomposition on $\SO(4)/\SO(3)$ 
must be 
naturally reductive. It is well-known that  
there is a one parameter family on naturally reductive decompositions 
on the Lie group $S^3 \simeq \Spin (3)$. 

The only reductive decomposition on the space $\SO (n+1)/\SO(n)$ 
is the orthogonal complement to $\so (n)$ in $\so (n+1)$, with 
respect to minus the Killing form of $\so (n+1)$. The reductive 
decompositions on $\SO (4)/\SO(3)\simeq \Spin (3)\times \Spin (3) /
\text {diag} (\Spin (3))$ are of one two types 
(cf. \cite {OlmosReggianiTamaru}, Section 5): 
\begin{enumerate}[(1)]
\item The orthogonal complement 
to $\text {diag} (\so (3))$ with respect to a bi-invariant 
pseudo-Riemannian (non-degenerate) 
scalar product on $\so (4) = \so (3) \oplus 
\so (3)$. Such an inner product has to be a multiple of minus 
the Killing 
form on each factor of $\so (4)$. These multiples, up to rescaling, 
are 
$\lambda _1 =1 $, 
$\lambda _2 \in \mathbb R$, $0\neq \lambda _ 2 \neq -1$. 
In this case the transvection group associated with the canonical connection 
is $\Spin (4)$. 
\item The reductive complement of $\text {diag} (\so (3))$ 
is either $(\so (3), 0)$ or $(0,\so (3))$. The trans\-vection group is either 
$\Spin (3)$, regarded as the left factor of $\Spin (4)$ or 
$\Spin (3)$, regarded as the right  factor of $\Spin (4)$. 
In both cases the canonical connection is flat. 
\end{enumerate}
\end{rem}

\begin{rem}\label {quotient}
Let $H$ be a connected  Lie subgroup of $\Spin (k+1)$ of
codimension $k \geq 2$. 
\begin{enumerate}[(i)]
\item If $k\neq 3$,  then $\Spin (k+1)/H$ is equivariantly
isomorphic to the sphere $S^k = \SO (k+1)/\SO(k)$. 
\item If $k=3$, then  $H$ is either  one factor of $\Spin (4) 
=\Spin (3)\times \Spin (3)$ or $\Spin (4)/H$ is equivariantly 
isomorphic to the sphere $S^3 = \SO (4)/SO(3)$.
\end{enumerate}
In fact, assume that no normal subgroup of 
$\Spin (k+1)$ with positive dimension is contained in the closure  $\bar H$ of
$H$. 
This is always the case if $k \neq 3$, since 
$\Spin (k+1)$ is a simple Lie group for $k \neq 3$. Note that 
$\bar H \neq \Spin (k+1)$, because otherwise
the Lie algebra of $\Spin (k+1)$ would 
have a flat factor. 
 Then $\Spin (k+1)$ acts 
almost effectively on the $k'$-dimensional compact quotient $M
 = \Spin (k+1)/\bar {H}$, where $0\leq k'\leq k$.
The manifold $M$ 
is simply connected since $\Spin (k+1)$ 
is simply connected and $\bar H$ is connected.
Since the dimension of the isometry group of $M$ is at 
least $k(k+1)/2$, then $M$ is isometric to a sphere,  
$k'=k$ and $\bar H = H$. 
Moreover, the effectivized action of $\Spin (k+1)$ 
gives the identity component of the full isometry group of the sphere, which is
isomorphic 
to $\SO (k+1)$. 
\end{rem}

\section{Classification for co-index of symmetry equal to 3}

Let $M= G/H$ be an $(r+3)$-dimensional ($r \geq 1$) compact simply connected
homogeneous
Riemannian manifold with coindex of symmetry $k = 3$. By
Theorem \ref{main} there exists a compact semisimple normal
subgroup $G'$ of $G$ with $\dim(G') \leq 6$ which acts transitively on $M$. We
may assume
that $G'$ is simply connected and that the action of $G'$ on $M$ is almost
effective. The only
possibilities for such a group are $G'= \Spin(4) = \Spin(3) \times
\Spin(3)$ and $G' = \Spin(3)$. However, since
 $M$ has a positive index of symmetry, we cannot have  $G' = \Spin(3)$.
Therefore $G' =
\Spin(4)$, which has dimension $6$, and so the dimension $d$ of
the isotropy group must satisfy $d \in \{0,1,2\}$. The case $d = 0$ can be
excluded from the last statement of Theorem \ref{mainco}.
If $d = 2$, then the isotropy group
is, up to conjugation, the standard torus $S^1 \times S^1 \subset \Spin(3)
\times \Spin(3)$. Such a quotient space, with any $G'$-invariant Riemannian
metric, is
the Riemannian product of two $2$-spheres. This implies that $M$ is symmetric
and so this
case can also be disregarded. 

We can therefore assume that the
dimension $d$ of the isotropy group $T$ is $1$. Thus $M$ is
$5$-dimensional and its index of symmetry is $2$. For such a
subgroup there are infinitely many possibilities, depending on the
different velocities of the projections of this subgroup to the
two factors. However, this is never the case when the index of symmetry is $2$
in which case we have the following lemma which uses the results of
the general theory we developed in Section 5. 

\begin{lem}\label{normal position} 
Let $M = \Spin(4)/T$ be a $5$-dimensional compact simply connected homogeneous
Riemannian manifold with coindex of symmetry $k = 3$.
Then, up to conjugation, $T = \diag(S^1) = \{(u, u) \in
  \Spin(3) \times \Spin (3): u \in S^1\}$. Moreover, after making
   the action effective, $M = \SO(4)/\SO(2)$, which is isometric to
  the unit tangent bundle of the $3$-sphere with an $\SO(4)$-invariant
Riemannian metric.
\end {lem}

\begin {proof}
We choose $p \in M$ such that $T$ is the isotropy group of $\Spin(4)$ at $p$.
Note that $T$ is connected since $M$ is simply connected. We consider $T$ as a
subgroup of $\SO(T_pM)$ via the isotropy representation of $M = \Spin(4)/T$ at
$p$. 
Since the distribution of symmetry $\mathfrak s$ is invariant under the action
of $\Spin(4)$ we see that $\mathfrak s_p$ is a $T$-invariant $2$-dimensional
subspace of $T_pM$.
We decompose $T_pM$ orthogonally into $T$-invariant subspaces,
 \[ T_pM = \mathfrak s_p \oplus \mathbb V \oplus \mathbb L, \]
where $\dim(\mathbb V) = 2$ and $\dim(\mathbb L) = 1$. Note that the action of
$T$ on $\mathfrak s_p$ or on $\mathbb V$ may be trivial. Let $\rho_1: T \to
\so(\mathfrak s_p)$, $\rho_1(h) =
  h_{\vert\mathfrak s_p}$ and let $\rho_2: T \to \so(\mathbb V)$,
  $\rho_2(h) = h_{\vert\mathbb V}$. It is not hard to see the
  following: \emph{if $\rho _1$ and $\rho _2$ are both (Lie group)
    isomorphisms, then $T$ is standard}. Namely, $T$ is conjugated to
  $\diag(S^1) = \{(h, h) \in \Spin(3) \times \Spin(3): h \in S^1\}$,
  where $S^1$ is any $1$-dimensional Lie subgroup of $\Spin (3)$.
  
  Let us show that both $\rho_1$ and $\rho_2$ are isomorphisms. Let
  $\Phi_i$ be the kernel of $\rho_i$, $i = 1, 2$. Since $T$ is
  abelian, then $\Phi_1$ and $\Phi_2$ are normal subgroups of the
  isotropy group $T$ at $p$.
  
  We first assume that $\Phi_1$ is not trivial.
  Then, in the notation of Proposition \ref{normal}, $\mathcal
  D^{\Phi_1}$ is the (unique) $\Spin(4)$-invariant autoparallel
  distribution with $\mathcal D^{\Phi_1}_p = \mathfrak s_p \oplus
  \mathbb L$. Due to Lemma~\ref{A} this distribution is strongly
  symmetric with respect to $\Spin(4)$. Moreover, $\mathfrak s$
  restricted to any integral manifold $F^{\Phi_1}(q)$ is a parallel
  distribution. Observe that the corank of $\mathcal D^{\Phi_1}$
  is $2$. Then, by Theorem \ref{main}, if $M$ does not split off a
  symmetric de Rham factor, $\dim(M) \leq 3$ (since there is
  $3$-dimensional group which is transitive on $M$). This is a contradiction and
hence $\Phi_1$ is trivial.
  
  We next assume that $\Phi_2$ is not trivial.
  Then, in the notation of Proposition \ref{normal}, $\mathcal
  D^{\Phi_2}$ is the (unique) $\Spin(4)$-invariant autoparallel
  distribution with $\mathcal D^{\Phi_2}_p = \mathbb V \oplus \mathbb
  L$. Observe that $\mathcal D^{\Phi_2} = \mathfrak s^\perp$. Since
  $\mathfrak s$ is also autoparallel, both distributions must be
  parallel and so $M$ splits off a symmetric space. This is a contradiction and
hence $\Phi _2$ is trivial.

It now follows that  $T$ is standard and so $M = \Spin(3) \times
  \Spin(3)/\diag(S^1)$. After making the action
  effective, this homogeneous space becomes $\SO(4)/\SO(2)$, where $\SO(2)$ is
naturally included
  in $\SO(4)$. So $M = \SO(4)/\SO(2)$, which is isometric to the unit tangent
  bundle of the $3$-sphere with a suitable $\SO(4)$-invariant Riemannian metric.
\end {proof}

We have proved that $M = \SO(4)/\SO(2)$. Let us determine the leaf of
symmetry at $p = [e]$. The subspace of vectors of $T_pM$ which are
fixed by the isotropy group $\SO(2)$ has dimension $1$. So the
$2$-dimensional leaf of symmetry $L(p)$ has non-trivial isotropy group. Thus
$L(p)$ is covered by a $2$-dimensional sphere and so the transvection group
$G^p$ is
$3$-dimensional (with Lie algebra isomorphic to $\so(3)$ and $K^p =
\SO(2)$). Since $\SO(2) \subset G^p$, $G^p$ cannot be contained in a
local factor of $\SO(4)$ (i.e., a factor corresponding to the
decomposition of $\Spin(4) = \Spin(3) \times \Spin(3)$). Then, by (ii)
of Remark \ref{quotient}, $\SO(4)/G^p$ is equivariantly isomorphic to
$\SO(4)/\SO(3)$. This isomorphism maps $\SO(2)$ into a $1$-dimensional
subgroup of $\SO(3)$. Such a group is conjugate in $\SO(3)$ to the
standard $\SO(2)$. Thus we may assume that $M = \SO(4)/\SO(2)$ and that
the leaf of symmetry at $p$ is given by
$$
L(p) = \SO(3)/\SO(2) \subset \SO(4)/\SO(2) = M.
$$

We have to determine  the $\SO(4)$-invariant metrics on $M =
\SO(4)/\SO(2)$ for which the index of symmetry is $2$.
As we observed above, the isotropy group $\SO(2)$ coincides with
the isotropy group $K^p$ of the transvection group $G^p =
\SO(3)$. In particular, since $K^p$ acts almost effectively on $L(p) =
\SO(3) \cdot p$ (see Facts \ref{facts}), we obtain that 
$$
H ^p := \{g \in G: g_{\vert L(p)} = \Id_{\vert L(p)}\}^o
$$
is trivial.

As we have noted before, if $\xi \in \so(4)$, regarded as a Killing
field of $M$, then there is $Z \in \mathfrak g^q $ such that $\xi -
Z$, restricted to $L(p)$, is always perpendicular to $L(p)$ (since the
projection of $\xi_{\vert L(p)}$ to $L(p)$ is an intrinsic
transvection of $L(p)$). Then, since $M$ is homogeneous, for any $u\in
(T_pL(p))^\perp$ there exists $\xi \in \so(4)$ such that $\xi . p
= u$ and $\xi$, restricted to $L(p)$, is always perpendicular to
$L(p)$. Moreover, such a $\xi$ is unique. In fact, assume that  $\eta
\in \so(4)$ is always perpendicular to $L(p)$ and $\eta.p =
0$. Then $\eta$ belongs to the isotropy algebra which coincides, as
previously observed, with $\mathfrak k^p$. Therefore $\eta$ is always
tangent to $L(p)$. It follows that $\eta_{\vert L(p)} = 0$ and so it belongs to
the Lie algebra $\mathfrak h^p$ of  $H^p$. This Lie algebra is trivial and thus
we have $\eta =0$.

Let 
$$
\mathfrak m = \{\xi \in \so(4): \xi_{\vert L(e)} \text { is always
  perpendicular to }L(p)\}.
$$
Then, since $L(p)$ is invariant under the action of $\SO(3)$, $\mathfrak m$ is
an
$\Ad(\SO(3))$-invariant subspace of $\so (4)$. Since the evaluation at
$p$, from $\mathfrak m$ into $(T_p(L(p)))^\perp$, is an isomorphism,
we obtain that
$$\so(4) = \so(3) \oplus \mathfrak m
$$
is a reductive decomposition of $\SO(4)/\SO(3)$ (the quotient space of
$M$ by the leaves of symmetry) and that 
$$
\mathfrak m.p = (T_p(L(p)))^\perp = (\so(3).p)^\perp.
$$ 
From Remark \ref{uniquereductive} we see that the above reductive
decomposition is naturally reductive (i.e., the canonical geodesics in
$S^3 = \SO(4)/\SO(3)$, associated to $\mathfrak m$, coincide with the
geodesics of the round sphere $S^3$)
and of one of the following forms:
\begin{enumerate}[(i)]
\item $\mathfrak m = \mathfrak m^\lambda$, where $\mathfrak m
  ^\lambda $ is the orthogonal complement of $\so (3)$ with respect to
  the (pseudo-Riemannian) inner product $(\,\,, \,)_\lambda = (B,
  \lambda B)$ of $\so(4) = \so(3) \oplus \so(3)$, $-B$ is 
  the Killing form of $\so(3)$ and $0 \neq \lambda \in \mathbb R$.
\item $\mathfrak m = \mathfrak m^0$, where $\mathfrak m^0 \simeq
  \so(3)$ is the Lie algebra of one of the factors of $\Spin(4)$ (and
  so $\mathfrak m^0$ is a Lie algebra). 
\end{enumerate}

We will now show that case (ii) cannot occur. Recall that, for
arbitrary Killing fields $\xi, \eta, X$, the Levi-Civita
connection is given by 
\begin{equation}\label{eq:case(ii)*}
  2\langle\nabla_\xi X, \eta\rangle = \langle[\xi, X], \eta\rangle +
  \langle[\xi, \eta], X\rangle + \langle[X, \eta], \xi\rangle  
\end{equation}
(see equation (3.4) of \cite{OlmosReggianiTamaru}). If $X \in \mathfrak p^p$ is a
transvection at $p = [e]$ and $\xi, \eta \in \mathfrak m^0$, then 
$0 = \langle[\xi, X], \eta\rangle_p + \langle[X, \eta], \xi\rangle_p$,
or equivalently, 
\begin{equation}\label{eq:case(ii)**}
  \langle[X, \xi], \eta\rangle_p = \langle[X, \eta], \xi\rangle_p .
\end{equation}
There exists $X \in \mathfrak p^p$ such that $[X, \mathfrak m^0] \neq
\{0\}$. Otherwise, $[\mathfrak p^p, \mathfrak m^0] = \{0\}$ and so
$[[\mathfrak p^p, \mathfrak p^p], \mathfrak m^0] = \{0\} $ and hence
$[\mathfrak g^p, \mathfrak m^0] = \{0\}$, which is  a contradiction (recall that
$\mathfrak g^p = \so(3)$, the Lie algebra of the standard $\SO(3)
\subset \SO(4)$, which is not an ideal of $\so(4)$). If we equip
$\so(4)$ with a bi-invariant (positive definite) metric, then $[X,
  \cdot\,]: \mathfrak m^0 \to \mathfrak m^0$ is skew-symmetric. Then
there exist linearly independent vectors $\xi, \eta \in \mathfrak m^0$ such 
that $[X, \xi] = \eta$ and $[X, \eta] = -\xi$. Inserting this into equation
(\ref{eq:case(ii)**}) leads to 
$\Vert\eta(p)\Vert^2 = -\Vert\xi(p)\Vert^2$,
which implies $\xi = 0 = \eta$ because, as previously observed, the
evaluation at $p$ is an isomorphism from $\mathfrak m^0$ onto
$(T_p(L(p))^p$. This is a contradiction and therefore case (ii) cannot occur.

We will now deal with case (i). For this we will use the construction 
given in \cite[Section 6]{OlmosReggianiTamaru}.

\underline{Case (a):} $\lambda > 0$, that is, the bi-invariant metric $(\cdot ,
\cdot )_\lambda = (B,
\lambda B)$ of $\so(4)$ is Riemannian. In the
notation of \cite{OlmosReggianiTamaru}, $G = \SO(4)$, $G' = \SO(3)$ and $K' = \SO(2)$
(and so $G\supset G' \supset K'$). Moreover, the general
assumptions in this reference are satisfied, i.e., $(\SO(4), \SO(3))$ and
$(\SO(3),
\SO(2))$ are irreducible symmetric pairs and $\SO(3)$ is a simple
(compact) Lie group. Let $\so(3) = \so(2) + \mathfrak p'$ be the
Cartan decomposition of $S^2 = \SO(3)/\SO(2)$. Since $\so(3)$ is
simple, the restriction of $(\cdot , \cdot)_\lambda$ to $\so(3)$ is a
multiple of the Killing form of $\so(3)$. So $\mathfrak p' \subset
\so(2)^\perp$ (the orthogonal complement in $\so(4)$ with respect to
$(\cdot , \cdot)_\lambda$), and thus
$$
\so(2)^\perp = \mathfrak m^\lambda \oplus \mathfrak p'.
$$ 

We will first define a Riemannian metric on $M = \SO(4)/\SO(2)$ such
that the canonical projection to the sphere $\SO(4)/\SO(3)$ is a Riemannian
submersion, with index of symmetry $2$ (and such that the orthogonal complement
to the subspace of symmetry is given by $\mathfrak m^\lambda \cdot
p$). Then we will deform this metric  to obtain all the
invariant metrics with index of symmetry $2$ and such that the
subspace which is orthogonal to the subspace of symmetry at $p = [e]$
is given by $\mathfrak m^\lambda. p$. 

Following \cite{OlmosReggianiTamaru}, we equip $T_p(\SO(4)/\SO(2)) \simeq
\so(2)^\perp = \mathfrak m^\lambda \oplus \mathfrak p'$ with the
positive definite inner product $\langle \cdot , \cdot \rangle_\lambda$ which is
defined by the following three properties: 
\begin{enumerate}[(i)]
\item $\langle\mathfrak m^\lambda, \mathfrak p'\rangle_\lambda = 0$;
\item the restrictions of both $(\cdot , \cdot)_\lambda$ and $\langle
\cdot , \cdot \rangle_\lambda$ to $\mathfrak m^\lambda$ coincide;  
\item $\langle \cdot , \cdot \rangle_\lambda = 2(\cdot , \cdot)_\lambda$
on $\mathfrak p' \times \mathfrak p'$. 
\end{enumerate}
We then equip $M = \SO(4)/\SO(2)$ with the $\SO(4)$-invariant metric, also
denoted by
$\langle\cdot , \cdot\rangle_\lambda$, which coincides at $p$ with the
above defined inner product. Then, by Lemma 6.2 in \cite{OlmosReggianiTamaru}, the
subspace of symmetry at $p$ is $\mathfrak p'.p$, unless $(M,
\langle\cdot , \cdot\rangle_\lambda)$ is symmetric (observe that $M$ is
simply connected). 

Since the fixed set of the isotropy representation of $\SO(2)$ on
$T_pM$ has dimension~$1$, it follows that the action of $\SO(2)$ on
$\mathfrak m^\lambda$ is non-trivial. Let $e_1, e_2, e_3$ be an
orthonormal basis 
of $\mathfrak m^\lambda \simeq \mathfrak m^\lambda.p$ with respect to
$\langle\cdot , \cdot\rangle_\lambda$. We may
assume, if $\mathbb R X_0 = \so(2)$, that $[X_0, e_1] = 0$, $[X_0,
  e_2] = e_3$ and $[X_0, e_3] = -e_2$. Observe that the isotropy group
$\SO(2)$ acts trivially on $\mathbb R e_1$ and irreducibly on the linear span
$\mathbb V$
of $e_2$ and $e_3$. Let $\langle\cdot , \cdot\rangle$
be an $\SO(4)$-invariant metric on $M = \SO(4)/\SO(2)$ such that $\mathfrak
m^\lambda.p$ is perpendicular to the subspace of symmetry
$\mathfrak p' .p = \so(3) . p$. Then, up to rescaling,
$\langle\cdot , \cdot\rangle$ has the following four properties:
\begin{enumerate}[(i)]
\item[(i)] $\langle\cdot , \cdot\rangle$ coincides with $\langle\cdot ,
\cdot\rangle_\lambda$ on $\mathfrak p' .p$;
\item[(ii)] $\langle e_1, \mathbb V\rangle = 0$;
\item[(iii)] $\langle e_1, e_1\rangle = s$ for some $s > 0$;
\item[(iv)] $\langle\cdot , \cdot\rangle = t\langle\cdot ,
\cdot\rangle_{\lambda}$ on
  $\mathbb V$ for some $t > 0$.
\end{enumerate}

We will now prove that $s + t = 2$.
Let $X \in \mathfrak p'$. Then $\SO(3) \cdot p$ is a totally geodesic
submanifold of $(M,\langle\cdot , \cdot\rangle)$
and $X_{\vert \SO(3) \cdot p}$ is an intrinsic transvection of $\SO(3)
\cdot p$ at $p$.
From equation (\ref{eq:case(ii)*}) we know that $X$ is a transvection at $p$
if and only if
\begin{equation}\label{eq:case(ii)***}
  \langle[\xi, X], \eta\rangle_p + \langle[\xi, \eta], X\rangle_p +
  \langle[X, \eta], \xi\rangle_p = 0
\end{equation}
holds for all $\xi, \eta \in \mathfrak m^\lambda$.
First of all, note that the orthogonal projection of $[e_2, e_3]$ onto $\so(3)$
 is a multiple of $X_0$. In fact, $[X_0, [e_2, e_3]] = [[X_0, e_2],
  e_3] + [e_2, [X_0, e_3]] = 0$. Now decompose $[e_2, e_3] = Z + \psi $ with $Z
\in \so(3)$ and $\psi \in \mathfrak m^\lambda$. Then $[X_0,
  Z] = 0$ and hence $Z = aX_0$, since $\so(3)$ has rank one (and so $Z.p = 0$). 
Next, we have
\begin{equation}\label{eq:formula-A}
  \begin{split} 
    2\langle\nabla_{e_1}X, e_2\rangle & = \langle[e_1, X],
    e_2\rangle_p + \langle[e_1, e_2 ], X\rangle_p + \langle[X, e_2],
    e_1\rangle_p  \\
     & = t\langle[e_1, X], e_2\rangle_{\lambda \vert p} + \langle[e_1,
      e_2], X\rangle_{\lambda \vert p} + s\langle[X, e_2],
    e_1\rangle_{\lambda \vert p}  .
  \end{split}
\end{equation}
The projection 
$\pi: (M, \langle\cdot , \cdot\rangle_{\lambda}) \to \SO(4)/\SO(3) = S^3$
is a Riemannian submersion, up to a rescaling of the metric. 
We denote by $\nabla^\lambda$ the Levi Civita connection of $M$ with respect to
$\langle \cdot , \cdot \rangle_\lambda$.
Since
$e_1$ and $e_2$ are projectable vector fields, which are horizontal along
$\SO(3) \cdot p$, we obtain
\[
  0  = (X\langle e_1, e_2\rangle_{\lambda})_p = \langle\nabla_X^\lambda
  e_1, e_2\rangle_{\lambda \vert p} + \langle e_1, \nabla_X^\lambda
  e_2\rangle_{\lambda \vert p} 
   = \langle[X, e_1], e_2\rangle_{\lambda \vert p} + \langle e_1, [X,
    e_2] \rangle _{\lambda \vert p},  
\]
because of  $[X, e_i]_p = (\nabla^\lambda_Xe_i)_p$ and since
$(\nabla^\lambda_{e_i}X)_p = 0$. Inserting this into
equation (\ref{eq:formula-A}) yields
\begin{equation}\label{eq:formula-B}
  2\langle\nabla_{e_1}X, e_2\rangle = (t + s)\langle[e_1, X],
  e_2\rangle_{\lambda \vert p} + \langle[e_1, e_2 ], X\rangle_{\lambda
    \vert p} .
\end{equation}
If $s = t = 1$ we have $\langle\nabla_{e_1}X, e_2\rangle = 0$ since $X$ is
parallel at $p$ because of $\langle \cdot , \cdot \rangle = \langle \cdot ,
\cdot \rangle_\lambda$ in this case. From equation (\ref{eq:formula-B}) we then
get $2\langle[e_1, X],e_2\rangle_{\lambda \vert p} = - \langle[e_1, e_2 ],
X\rangle_{\lambda \vert p}$ in this case.
We have that $[\mathfrak m^\lambda, \mathfrak m^\lambda]^{\so(3)} =
\so(3)$, where $(\,\,)^{\so(3)}$ denotes the projection onto
$\so(3)$. In fact, this projection is not trivial, since $\mathfrak
m^\lambda$ is not a Lie algebra and
$\Ad(\SO(3))$-invariant. Recall, as we have shown, that $[e_2,
  e_3]^{\so(3)}\subset \so(2)$. Then $[e_1, e_2]$ 
projects non-trivially into $\mathfrak p'$. 
If $X$ would be parallel at $p$, for any $X$ in $\mathfrak p ' $, then we would
also have that 
$(s+t) \langle[e_1, X],e_2\rangle_{\lambda \vert p} = - \langle[e_1, e_2 ],
X\rangle_{\lambda \vert p}$
for any $X\in \mathfrak p '$, which implies that $ - \langle[e_1, e_2 ],
X\rangle_{\lambda \vert p} = 0$.
In particular, for $X$ equal to the projection to $\mathfrak p '$ of 
$[e_1 , e_2]$, this gives a contradiction. 
This implies that $X$ is a
transvection of $(M,\langle\cdot , \cdot\rangle)$ at $p$ if and only if $t
= 2 - s$, $0 < s < 2$. 

We denote this metric by $\langle\cdot , \cdot\rangle_{(\lambda, s)}$ with $0 <
\lambda$ and $0 < s < 2$. If we replace $\lambda$ by
$1/\lambda$  the metrics are homothetical, so we may assume that
$0 < \lambda \leq 1$ (see Remark \ref{lambda}).

\underline{ Case (b):} $\lambda < 0$, that is, $(\cdot , \cdot)_\lambda = (B,
\lambda
B)$ is a pseudo-Riemannian bi-invariant metric on $\so(4)$. By making
the same construction as in Case (a), eventually by changing the sign
of the metric, we obtain a pseudo-Riemannian metric $\langle\cdot ,
\cdot\rangle_\lambda$ on $M$ such that it is positive definite on $\so(3)
.p$ and negative definite on its orthogonal complement $\mathfrak
m^\lambda.p$. Moreover, if $X \in \mathfrak p' .p$, then
$(\nabla^\lambda X)_p = 0$. As in Case (a), such a metric can only be deformed
when rescaling by $s$ on $\mathbb R e_1$ and by
$2 - s$ on  $\mathbb V$ (in order that $X$ is a transvection at
$p$). But $s$ and $2 - s$ cannot be both negative in order for the
metric $\langle\cdot , \cdot\rangle_{(\lambda ,s)}$ to be Riemannian. So this
case can be excluded.  

We conclude that, if the index of symmetry of $\SO(4)/\SO(2)$ is $2$, then the
Riemannian metric has to be of the form $\langle\cdot , \cdot\rangle_{(\lambda
,s)}$ with $0 <\lambda$ , $0 < s < 2$.

Conversely, such metrics have index of symmetry $2$, unless the space
is globally symmetric. In fact, the distribution of symmetry on
$\SO(4)/\SO(2)$ descends to a $\SO(4)$-invariant (and therefore parallel)
distribution on the irreducible symmetric space $S^3 =
\SO(4)/\SO(3)$. Such a distribution must be trivial, and if the rank is zero the
index of symmetry of $\SO(4)/\SO(2)$ is $2$, and if the rank is maximal then
$\SO(4)/\SO(2)$ has index of symmetry $5$ and so it is a symmetric
space. 

\begin{rem}\label{lambda}
  Let us consider the bi-invariant inner product $(B,\lambda B)$, $ \lambda >0 $
  on $\so(4) = \so(3) \oplus \so(3)$, where $-B$ is the Killing form
  of $\so(3)$. The involution $\tau$ of $\Spin(4) =
  \Spin(3) \times \Spin(3)$, that permutes the factors, maps
  both $\diag(\SO(3))$  and $\diag(\SO(2))$ into itself. So
  $\tau$ induces an isomorphism $\bar\tau$ of $M =
  \Spin(4)/\diag(\Spin(2))$ into itself. The map $\bar\tau$ is an
  isometry from $(M, \langle\,\,, \,\rangle)$ into $(M, \langle\,\,,
  \,\rangle')$, where $\langle\,\,, \,\rangle$ is the normal
  homogeneous metric with respect to $(B, \lambda B)$ and
  $\langle\,\,, \,\rangle'$ is the normal homogeneous metric with
  respect to $(\lambda B, B)$. The same is true if we rescale the
  metrics by a factor $2$, as in our construction, on the tangent
  space of $\diag(\Spin(3))/\diag(\Spin(2))$ at $[e]$. Now observe
  that the normal homogeneous metric on $M$ with respect to $(\lambda
  B, B)$, or that modified as before, is homothetical to the normal
  homogeneous metric induced by $(B, \frac{1}{\lambda} B)$. 
\end {rem}
 
\begin{rem}\label{25-10} 
  A compact, simply connected, Riemannian symmetric space of dimension $5$ is isometric to
  one of the following spaces: $S^2 \times S^3$, $S^5$ or
  $\SU(3)/\SO(3)$. The last space is irreducible and of rank $2$.

 The homogeneous space $\SO(4)/\SO(2)$ is not homeomorphic to $S^5$. In
    fact, from the long exact homotopy sequence of the
    fibration $\SO(2) \to \SO(4) \to \SO(4)/\SO(2)$
    it follows that $\pi_3(\SO(4)/\SO(2)) = \mathbb Z \oplus \mathbb
    Z \neq \pi_3(S^3)$. 

The space $M^5 = \SO(4)/\SO(2)$, with any $\SO(4)$-invariant metric, can
    never be isometric to an irreducible symmetric space of higher
    rank. In fact, if $p = [e]$, the isotropy representation of
    $\SO(2)$ on $T_pM$ is the direct sum of two copies of the standard
    representation of $\SO(2)$ on $\mathbb R^2$, plus a trivial one-dimensional
representation. If $\phi \in \SO(2)$ is the
    rotation of angle $\pi$ (with the standard representation), then
    $\phi$ represents an element of the isotropy group of $M$ which has the
    eigenvalue $-1$ with multiplicity $4$ and the eigenvalue $1$ with
multiplicity $1$. If
    $M$ is a symmetric space, then the decomposition of $\phi$ with respect to
the
    symmetry $\sigma$ at $p$, via the isotropy representation, has the
    eigenvalue $1$ with multiplicity $4$ and the eigenvalue $-1$ with
multiplicity $1$. Then
    the connected component containing $p$ of the fixed set of $\sigma \circ
    \phi$ would be a totally geodesic hypersurface $N$ of $M$. Let
    $K'$ be the full connected isotropy group of $N$ at $p$. We may regard
    $K' \subset K$, where $K$ is the full connected isotropy group of the
    symmetric space $M$. Observe that $K'$, via the isotropy
    representation, acts trivially on the one-dimensional normal space
    $\nu_p(N) \simeq \mathbb R$ of $N$ at $p$. Let $\bar R$ be the
    direct product of $R'$ and the zero tensor on $\nu_p(N)$, where
    $R'$ is the curvature tensor of $N$ at $p$. Then $\bar R_{x, y}
    \in \mathfrak k$ and so, by Simons' Theorem \cite{Olmos,Simons}, if $M$ is of rank
    at least $2$, $\bar R$ must be a scalar multiple of $R$, the
    curvature tensor of $M$ at $p$. This is a contradiction if $M$ is an
    irreducible symmetric space. Thus $M$ cannot be isometric to the
    irreducible rank $2$ symmetric space $\SU(3)/\SO(3)$. 

Note that $\SO(4)/\SO(2)$ is diffeomorphic to $S^2 \times
    S^3$, since the first space is diffeormorphic to the unit tangent
    bundle of the (parallelizable) sphere~$S^3$. 
\end {rem}

\begin{exa}\label{productofspheres}
 ({\bf Product of spheres})
We denote by $S^2$ the sphere of dimension $2$ and radius $\rho$ and by
$S^3 $ the sphere of dimension $3$ and radius $1$, and put
$M = S^2 \times S^3$. Observe that any product of a round $2$-sphere and a
round $3$-sphere is homothetic to $M$ with a suitable $\rho$. 

The group $\Spin(4) = \Spin(3) \times \Spin(3)$ acts transitively by
isometries on $M = S^2 \times S^3 \simeq S^2 \times\Spin (3)$ in the
following way: 
$$
(g, h) ((q, k) )= (\pi(g)(q), gkh^{-1}),
$$ 
where $(g, h) \in \Spin(3) \times \Spin (3)$, $q \in S^2$, $k \in
\Spin(3) \simeq S^3$, and $\pi$ is the canonical projection from $\Spin(3)$ onto
$\SO(3)$. The isotropy group at $p = (\rho e_1, e) \in S^2 \times \Spin(3)$
is $\diag(\SO(2)) \subset \Spin(3) \times \Spin(3)$. After making this
action effective, one obtains that $\SO(4)$ acts transitively on $M$
and the isotropy group is conjugate to $\SO (2)$, where $\SO(2) \subset
\SO(4)$ is the standard inclusion.
Recall that for $\so(n)$ the Killing form $-B$ is given by
$$
B(X,Y) = -(n - 2)\trace(X\circ Y).
$$
For $n = 3$ the Killing form coincides with the negative of the usual inner
product of matrices.

Let $p = (\rho e_1, e) \in M = S^2 \times \Spin(3)$, where $e_1 = (1,
0, 0)$. The parallel Killing fields at the identity $e$ of $\Spin(3) =
S^3$ are the elements of $\so(3) \times \so(3)$ of the form $Z = (X,
-X)$ (regarded as a Killing field on $\Spin(3)$). The parallel Killing
fields on $S^2$ at $\rho e_1$ are elements in the Cartan subspace 
$$
\mathfrak p = \left\{\left(
\begin{smallmatrix} 
  0 & a & b \\
  -a & 0 & 0 \\ 
  -b & 0 & 0 
\end{smallmatrix}\right): a,b \in \mathbb R\right\}
$$ 
associated with the symmetric pair $(\SO(3), \SO(2))$.
Therefore an element $Z \in \so(3) \times \so(3)$ is parallel at $(\rho
e_1, e)$ if and only if $Z = (Y, -Y)$ with $Y \in \mathfrak
p$. Observe that the subspace $\mathfrak p^{(\rho e_1, e)} = \{(Y,
-Y): Y \in \mathfrak p\}$ of parallel Killing fields at $(\rho e_1, e)
\in S^2 \times \Spin(3)$ belonging to $\so(4) = \so(3) \oplus
\so(3)$ has dimension $2$. We use here the general notation of the
paper, but take into account that the Cartan subspace is relative to
the presentation group (i.e., the parallel Killings field at a given
point that lie in the Lie algebra $\so(3) \times \so(3)$).  
The (relative) Cartan subspace is given by $\mathfrak p^{(\rho e_1,
  e)}$, which spans the involutive Lie algebra
$$
\mathfrak g^{(\rho e_1, e)} = \diag(\so(2)) \oplus \mathfrak p^{(\rho
  e_1, e)}, 
$$
where $\so(2) = \{u \in \so(3): u \cdot e_1 = 0\}$.

Up to homothety, $S^2 \times S^3$ must carry a
metric $\langle\cdot , \cdot\rangle_{(\lambda , s)}$ as described
above (recall that $\rho$ is the radius of $S^2$ and $1$ is the
radius of $S^3$). 
We will now determine $\lambda$. Observe that $G^{(\rho e_1, e)}$, the group
which is generated by the transvections at $(\rho e_1, e)$, is not the
canonical $\diag(\Spin(3)) \subset \Spin(3) \times \Spin(3)$ (but it
must be conjugate to it). So the reductive complement, associated to
the Killing fields in $\so(3) \times \so(3)$ that are always
perpendicular to $L((\rho e_1, e)) = G^{(\rho e_1, e)} \cdot (\rho
e_1, e)$, is conjugate to $\mathfrak m_\lambda = \{(Z,
-\frac{1}{\lambda}Z): Z \in \so(3)\}$. 

We will find $h \in \Spin(3)$ such that $G^{(\rho e_1, h)} =
\diag(\Spin(3))$. In order to simplify the calculations, we will use the
quaternions. Identify $\Spin(3)$ with the unit sphere of the
quaternionic space $\mathbb H = \{a + ib + cj + dk: a, b, c, d \in
\mathbb R\}$, $i^2 = j^2 = k^2 = -1$, $ij = -ji = k$, $jk = -kj = i$,
$ki = -ik = j$. Let  $\pi: \Spin(3) \to \SO(3)$ be the canonical
projection. By identifying $\mathbb R^3$ with the purely imaginary quaternions $\Im(\mathbb H) = \{q \in H: \bar q = -q\} $ we obtain 
$$
\pi(g)(x) = gxg^{-1} = gx \bar g.
$$ 
The Lie algebra $\so(3)$ of $\Spin (3)$ is identified with
$\Im(\mathbb H)$ with the bracket $[x, y] = xy - yx$. Observe that,
with these identifications, $i = e_1$, $1 = e$.  
The exponential map is given by $\Exp(x) = \cos(\Vert x\Vert ) +
\sin(\Vert x\Vert )\frac{1}{\Vert x\Vert} x$. If $x \in \Im(\mathbb
H)$, then $\frac{d}{dt}_{\vert t=0} \pi(\Exp(tx)(z) = xz - zx$. So $x$
defines the Killing field of $\Im(\mathbb H)$ given by $z \mapsto x
.z = xz - zx$.  
Observe that 
$$
\so (2) = \{U \in \so(3): U.e_1 = 0\} = \{w \in \Im(\mathbb H):
wi - iw = 0\} = \mathbb R i.
$$ 
With these identifications the (relative) Cartan subspace $\mathfrak p$ is given
by 
the linear span of $j$ and $k$.
It is not hard to see that 
$(1, -i)G^{(\rho i, 1)}(1, -i)^{-1} = \diag(\Spin(3))$
and thus
$$
G^{(\rho i, i)} = G^{(1, -i) \cdot (\rho i, 1)} = (1, -i)G^{(\rho i,
  1)}(1, -i)^{-1} = \diag(\Spin(3)).
$$
Moreover, $\mathfrak k^{(\rho i, i)} = \mathbb R i$ and
$$
\mathfrak p^{(\rho e_1, i)} = \diag(\mathfrak p) = \{(Y, Y): Y \in
\mathfrak p\} = \{(v, v): v \in \text{linear span of } \{j,k\}\}.
$$ 
If $v \in \mathfrak p$, then 
$$
(v, v).(\rho i, i) = (v .\rho i, v. i) = (\rho(vi
- iv), vi - vi) = (2\rho vi, 2vi).
$$
Observe that $vi \in \mathfrak p$ and therefore 
$$
\mathfrak s^{(\rho i, i)} = \mathfrak p^{(\rho i, i)} . (\rho i, i) =  
\{(\rho v, v): v \in \mathfrak p\}.
$$ 
This subspace must be perpendicular to $\mathfrak m^\lambda .
(\rho i, i)$, where
$$
\mathfrak m^\lambda = \{(Z, -\tfrac{1}{\lambda}Z): Z \in \so(3) =
\Im(\mathbb H)\}.
$$
Take $Z = k, Y = j \in \mathfrak p$. Then $(k, -\frac{1}{\lambda}k)
. (\rho i, i) = (2\rho j, (1 - \frac{1}{\lambda})j)$. This must be
perpendicular to $(\rho j, j)$. Then $2\rho^2 = \frac{1}{\lambda} - 1$
and therefore 
$$
\lambda = \frac {1}{1 + 2\rho ^2}.
$$
 
The fixed vectors in $\mathfrak m^{\frac{1}{1 + 2\rho^2}}$ are
$\mathbb R(i, -(1 + 2\rho^2)i) \in \so(3) \oplus \so(3)$. Let us
compare the metric on the product of spheres with the one given by the
bi-invariant inner product $(B, \frac{1}{1 + 2\rho^2}B)$.  
The  norm of $(i, -(1 + 2\rho^2)i)$ with the given metric is 
\begin{align*}
  \Vert(i, -(1 + 2\rho^2)i) . (\rho i, i)\Vert^2  & = \Vert([i, \rho
    i], ii + i(1 + 2\rho^2)i)\Vert^2 \\
   & = \Vert (0, -2(1 + \rho^2)\Vert^2 
   = 4(1 + \rho^2)^2,
\end{align*}
and the norm, using $(B, \frac{1}{1 + 2\rho^2} B)$, is 
\begin{align*}
  \Vert(i, -(1 + 2\rho^2)i)\Vert^2 & = B(i, i) + \frac{1}{(1 +
    2\rho^2)}B(-(1 + 2\rho^2)i, -(1 + 2\rho^2)i) \\
  & = (8 + 8(1 + 2\rho^2)) =
   16(1 + \rho^2),
\end{align*}
since $B(i, i) = 8$. So the quotient is 
$s' = \frac{1}{4}(1 + \rho^2)$.

Let us choose the element $(j, -(1 + 2\rho^2)j) \in \mathfrak m^{\frac{1}
{(1 + 2\rho^2)}}$ that is perpendicular to the fixed vectors $\mathbb
R(i, -(1 + 2\rho^2)i)$ of the isotropy group. The norm with the given metric
is 
\begin{align*}
  \Vert(j, -(1 + 2\rho^2)j) . (\rho i, i)\Vert^2 & = \Vert([j,
    \rho i], ji + i(1 + 2\rho^2)j)\Vert^2 \\ 
  & = \Vert -2\rho k, 2\rho^2k)\Vert^2 \\
  & = 4\rho^2 + 4\rho^4 = 4\rho^2 (1 + \rho ^2),
\end{align*}
and the norm using $(B,\frac{1}{(1 + 2\rho^2)}B)$ gives, as before,
$$
\Vert(j, -(1 + 2\rho^2)j)\Vert^2 = 16(1 + \rho^2).
$$
The quotient is $t' = \frac{1}{4}\rho^2$.

We have $s' + t'\neq 2$ because we need to rescale the metric in line with our
classification. So, define $s = \frac{2s'}{s' + t'}$, and the metric
$\langle\cdot,\cdot\rangle_{(\frac{1}{1 + 2\rho^2}, s)}$ is the metric in the
family. 
An explicit calculation gives
$$
s = 2\frac{1 + \rho^2}{1 +2\rho^2} \ {\rm and}\ t = 2\frac{\rho^2}{1
  +2\rho^2}.
$$
For instance, if $\rho = 1$, then $\lambda = \frac{1}{3}$, $s =
\frac{4}{3}$ and $t = \frac{2}{3}$.
\end{exa}

\begin{rem}\label{lambda + 1}
  Recall that, in the above examples of products of spheres, $\lambda =
  \frac{1}{1 + 2\rho^2}$ and $s = 2\frac{1 + \rho^2}{1 +
    2\rho^2}$. Then $s = \lambda + 1$. Therefore the  family of
  examples of products of spheres as previously discussed 
  corresponds to the family of metrics
  $\langle\cdot , \cdot\rangle_{(\lambda, \lambda + 1)}$, where $ 0 <
  \lambda <1$ (and the quotient of the radius of the $2$-sphere by the
  radius of the $3$-sphere is given by $\rho = \sqrt{\frac{1 -
      \lambda}{2\lambda}}$).
  In particular, the reductive complement is never the standard one,
  i.e., $\lambda \neq 1$. Observe also that $0 < t < s < 2$ (recall
  that $s + t = 2$). Then the metric does not project down, as a
  Riemannian submersion, to the quotient $\SO(4)/\SO(3)$ of $M$ by the
  leaves of symmetry (relative to $\SO(4)$).  
\end {rem} 

\begin{rem}\label{otra}
  Any transitive action of $\Spin(3) \times \Spin(3)$ on $S^2 \times
  S^3 \simeq S^2 \times \Spin(3)$ is equivalent to the previously
  described action or to the  action given by
  $$
  (g, h) ((u, d) )= (\pi(g)(u), h (d)).
  $$ 
  However, the isotropy group of the latter action is $\SO(2) \times \{e\}$ and
  fixes the $3$-dimensional space $T_d(\Spin(3))$. So this
  homogeneous space is not (equivariantly) isomorphic to the canonical
  $\SO(4)/\SO(2)$.
\end {rem}

We can now state the main result of this section.

\begin{thm}\label{classification}
  Let $M$ be an $n$-dimensional, simply connected, compact,  irreducible
  Riemannian homogeneous manifold and $n > 3$. Then the co-index of symmetry of
  $M$ is equal to $3$ if and only if $M$ is homothetic to $M =
  \SO(4)/\SO(2)$ with a metric of the family 
$\langle\cdot , \cdot\rangle_{(\lambda, s)}$, where $0 < \lambda \leq 1$,
  $0 < s < 2$ and $s \neq \lambda +1$. (If $s = \lambda + 1$, then, up to
  homothety, $M$ is a product of spheres $S^2_{\rho} \times S^3$ with
  $\rho = \sqrt{\frac{1 - \lambda}{2\lambda}}$.)
\end {thm}

\begin {proof}
It only remains to prove that different pairs 
$(\lambda , s)$ correspond to non-homo\-thetical metrics.
First of all, we note that $\SO (4)$ is the (connected) full isometry group of 
$ M = \Spin (4) / \text {diag} (\SO (2))= \SO (4)/\SO(2)$ 
with any of the metrics
of the family $\langle\cdot , \cdot\rangle_{(\lambda, s)}$. 
(Note that $M$ is not symmetric.) Otherwise, by Remark \ref {25-10}, 
it would be a product of spheres. But such a product of spheres 
corresponds to $s=\lambda + 1$ (see Remark \ref {otra}).
So, by the paragraph before Remark \ref  {lambda}, the index of symmetry 
of $M$ is $2$.
 So, in the $3$-dimensional 
quotient $N$ of $M$ by the leaves of symmetry, the group $SO(4)$ acts by 
isometries (with the normal homogeneous metric). Then, up to a cover, 
$N$ is a sphere 
and hence $\SO (4)$ must be the full (connected) isometry group of $N$. 
Therefore,  if the  isometry group $I^o (M)$ of $M$ is bigger than 
$\SO (4)$, then $I^o (M)$ has a  proper (connected) normal 
subgroup $H$ acting trivially on $N$. If $L([e])
= \SO (3)/\SO (2)\simeq S^2$ is the leaf of symmetry at 
$[e]$, then $H\cdot L([e])= L([e])$ and $H$ commutes with 
$\SO (3)$, which is a contradiction. Hence we must have
$I^o (M) = \SO (4)$. 

Let us assume that the pairs $(\lambda , s)$ and $(\lambda ' , s')$ 
correspond to homothetical metrics (and the pairs do not correspond to 
the exceptions that are product of spheres). 
Assume that $\lambda \neq \lambda '$, say
  $\lambda < \lambda '$.
If $h$ is the homothety  
between the metrics, then it induces a Lie algebra isomorphism
$\rho = h _*$
of $\so (4)$  (the Lie algebra of the full isometry groups) 
 that maps 
$\text {diag} (\so (3))$ into itself (since it corresponds to the 
group of transvections at $[e]$) and  $\rho$ maps  
$\text {diag} (\SO (2))$ into itself (the Lie algebras of the isotropy at
$[e]$).
Moreover,   $\rho ( \mathfrak m ^\lambda)  =
\mathfrak m ^{\lambda '} $. 
In fact, these subspaces are given by the geometry 
as the Killing fields 
which are always perpendicular to the leaves  of 
symmetry $\SO (3)/\SO (2) = 
\text {diag} (\SO (3))/\text {diag} (\SO (2))$, with the respective 
metrics. 
 Observe that $\rho$ must preserve 
$(B,B)$, where $-B$ is the Killing form of $\so (3)$.
Let $(u, 0)\in \so (3)\oplus \so (3) = \so (4)$.
Then $$(u,0) = \frac 1 {1 + \lambda } (u,u) +
\frac \lambda {1 + \lambda}(u, -\frac 1 \lambda u),$$
which gives the decomposition of $(u,0)$ in terms of 
the direct sum
$$\so (3) \oplus \so (3) = \text {diag}(\so (3)) \oplus 
\mathfrak m ^{\lambda}.$$
Then the projection to $\text {diag}(\so (3))$
is given by 
$$\pi ^\lambda ((u,0)) = \frac 1 {1 + \lambda } (u,u).$$
We  also have that 
$$(0,v) = \frac \lambda {1+\lambda}(v,v) -
\frac \lambda {1+\lambda}(v, -\frac 1 \lambda v)$$
and so 
$$\pi ^\lambda ((0,v)) = \frac \lambda {1 + \lambda}
(v,v).$$
Since $\rho (\text {diag}(\so (3))) = \text {diag}(\so (3))$ 
and 
$\rho (\mathfrak m ^{\lambda})= \mathfrak m ^{\lambda '}$, we obtain 
that 
$$\rho \circ \pi ^\lambda = \pi ^{\lambda '}.$$
Since $\rho :\so (3)\oplus \so (3) \to \so (3)\oplus \so (3)$ 
is a Lie algebra isomorphism, $\rho ((u,0))$ is either of the 
form $(u',0)$ or $(0,u')$. Moreover, since $\rho$ preserves the 
Killing form,  $B(u,u) = B(u',u')$. 
Also ,
\begin{eqnarray*}
B ( \pi ^{\lambda '} (\rho ((u,0))), \pi ^{\lambda '}(\rho ((u,0))))
 & = & B(\rho (\pi ^\lambda ((u,0))), \rho (\pi ^\lambda ((u,0)))) \\
&= & B(\pi ^\lambda ((u,0)), \pi ^\lambda ((u,0))) .
\end{eqnarray*}
Let us choose $u\neq 0$. If $\rho ((u,0)) 
= (u', 0)$ we have, from  the above equality,  that 
$$\frac 1 {1+\lambda ' }B(u',u') = 
 \frac 1 {1+\lambda  }B(u,u),$$ 
and so $ 1+\lambda '  = 1+\lambda  $. 
This is a contradiction to $\lambda \neq \lambda  '$. 
If $\rho ((u,0)) = (0,u')$, then the previous equality implies 
$\frac {\lambda '}{ 1+\lambda '} = 
\frac {1}{ 1 + \lambda}$,
which gives also a contradiction, since $0<\lambda <\lambda ' \leq  1$. 
It follows that  $\lambda = \lambda '$.

Since the curvature of the leaf of symmetry $\SO(3)/\SO(2)$ 
of $\SO (4)/\SO (2)$  with respect to the metric 
$\langle \cdot , \cdot \rangle_ {(\lambda , t)}$ depends only 
on $\lambda$ (and $B$), and since the homothety $h$ maps leaves of symmetry onto
leaves of symmetry, we see
that the homothety must be an isometry. 
We choose $v$ in $\mathfrak m ^{\lambda}$ of unit length and 
fixed by the isotropy group. 
Then the length of the closed geodesic $\gamma _v (t)$ determined by $v$ is
equal to
$as$, where $a$ is a constant. Since $h$ maps $\mathfrak m ^{\lambda}$
onto $\mathfrak m ^{\lambda  '}$ and fixed vectors of the isotropy group
onto fixed vectors of the isotropy group, $h (\gamma _v (t)) = 
\gamma _{v'}(t)$, where $dh (v) = v' $. Since the second geodesic 
has length $as'$, then $s=s'$. 
\end {proof}

\section {Classification for co-index of symmetry equal to 2}

The main result of this section is the following classification:

\begin {thm}\label {dos} 
Let $M$ be an $n$-dimensional ($n > 2$), simply connected, compact, 
 irreducible Riemannian
homogeneous manifold with co-index of  symmetry $k = 2$. Then 
$M = \Spin(3)$ with a left-invariant Riemannian metric that
belongs to one of the two families 
$\langle \cdot , \cdot \rangle _{s}$ ($0<s< 1$) and 
$ \langle \cdot , \cdot \rangle^{t}$ ($0<t\neq 2$) which 
are described below.   
None of these metrics are pairwise homothetic. The second family of metrics
corresponds 
to Berger sphere metrics.  
\end {thm}

The rest of this section is devoted to the proof of Theorem \ref{dos}.
If $M$ is a homogeneous irreducible Riemannian manifold with co-index  
of symmetry $k=2$, then 
$M = \Spin (3)$ with a left-invariant Riemannian metric  by  Theorem \ref
{mainco}. 

Let us first describe  the left-invariant Riemannian metrics on
$\Spin (3)\simeq S^3$.
 As usual, we will  identify a
 left-invariant Riemannian metric on $\Spin(3)$
 with a positive definite  inner product on
 $T_e(\Spin (3)) \simeq \so (3) $.
 Let $B $ be the positive definite inner product  on
 $\so (3)$ given by
 $B(X,Y) =
 - \text {trace} (XY)$ (so $-B$ is the Killing form of $\so (3)$).
Any  positive definite inner product $\langle \cdot , \cdot \rangle$
on $\so (3)$ is obtained by
 $\langle X ,Y\rangle     =   B( AX ,  Y) $,
 where $A$ is a positive definite symmetric endomorphism,
 with respect to $B$, of
 $\so (3)$.
 Observe that any positive definite inner product $\langle X , Y \rangle =
 B( AX ,  Y )$ is isometric  to the inner product
 $$ B (A(\Ad (g) (X)) ,
 \Ad(g) (Y)) =
  B( (\Ad (g))^{-1} A(\Ad (g)) (X) ,
  Y), $$
  for any $g\in \Spin (3)$ (the isometry between the corresponding two
left-invariant Riemannian metrics
is given by conjugation with $g$ in $\Spin (3)$).
  Note that
  $\Ad(\Spin (3))$ coincides with  
  the full special orthogonal group $\SO (\so (3), B)$.
Then, for prescribing an arbitrary left-invariant Riemannian metric on $\Spin
(3)$ (modulo 
isometries) one only needs to know the eigenvalues of $A$. 

 We identify $X\in \so(3)$  with the Killing
 field $q \mapsto X.q = \frac {d}{dt}_{\vert t=0}\Exp (tX)(q)$. The Lie algebra
structure  on $\so(3)$ will be that of Killing fields. 
So the Lie bracket is given by
  $[X, Y] = XY-YX $,
  which is minus the  bracket of left-invariant vector fields, 
since a Killing field may be regarded as a right-invariant vector field.

Let $\mathfrak s$ be the $1$-dimensional 
distribution of symmetry on $\Spin (3)$. Since  $\mathfrak s$ is 
a left-invariant distribution, we may assume that 
$\mathfrak s _1 = \mathbb R i$, where we are using, 
as before, the quaternions. We identify $\Spin (3)$ 
with the unit sphere of $\mathbb H$ and $\so (3)$ 
with $\text {Im}(\mathbb H)$. With this identification the bracket of 
$q_1,q_2 \in \text {Im}(\mathbb H)$  is given by 
   $q_1q_2-q_2q_1$, which coincides with $-[q_1,q_2]$, where  
$[\cdot , \cdot]$ is  the bracket between  Killing fields of 
$(\Spin (3) \langle \cdot , \cdot\rangle)$  
(identifying 
$q\in \text {Im}(\mathbb H)$ with the Killing field $x\mapsto q.x$).
The Killing form $-B$ is given by $B (q,q) = 8\vert q\vert ^2 $, 
$ q \in  \text {Im}(\mathbb H)$.

As for the case $k=3$,  we define
$$\mathfrak m = \{q\in \text {Im}(\mathbb H): q \text { 
is always perpendicular to } L(1) = \text {e}^{ti}\}.$$
Then $\mathfrak m $ is an $\Ad (S^1)$-invariant subspace of 
$\text {Im}(\mathbb H)\simeq \so (3)$, where $S^1= 
\{\text {e}^{ti} : t\in \mathbb R\}$.
Then, by Remark  \ref {uniquereductive} , $\mathfrak m$ is unique and 
so it coincides with  the linear span of $\{j,k\}$. This implies 
that the vectors $j= j.1$ and $k= k.1 $ of $T_1(\Spin (3))$ are 
perpendicular to $\mathfrak s _1 = \mathbb R i$. So 
$\langle i ,j\rangle = 0 = \langle i , k\rangle $. Then, if 
$\langle q, q' \rangle = B(Aq, q')$, $i$ is an eigenvector of 
$A$. By conjugating $\Spin (3)$ with some 
$\text {e}^{ti}$, we may assume that $j$ and $k$ are also eigenvectors 
of $A$. By rescaling the metric $\langle \cdot , \cdot \rangle$ we may assume 
that $Ai = 2i$ (in order to use a similar   
construction as for the case $k=3$, where the normal 
homogeneous metric was at the first step 
perturbed by a factor $2$ on the distribution of symmetry). 
Let $Aj=s j$ and $Ak= t k$. We may assume 
that $0< s \leq  t$ (eventually, by conjugating $\Spin (3)$ with $i$).
We will now consider $i$, $j$ and $k$ as  Killing fields 
$I : q\mapsto i.q$, $J : q\mapsto j.q$ and $K : q\mapsto k.q$.

We first assume that $I^o(\Spin (3) , \langle \cdot , \cdot \rangle) = 
\Spin (3)$. In this case we have 
$(\nabla I)_1 = 0$, since there are no more Killing fields than those 
induced by  $\so (3)$.
Recall 
 that for any homogeneous Riemannian manifold, if $X,Y,Z$ are Killing
fields, then the Levi-Civita connection 
is given by 
$$2 \langle \nabla _X Y , Z \rangle = \langle [X,Y]  , Z \rangle
+ \langle [X,Z]  , Y \rangle + \langle [Y,Z]  , X \rangle.
$$
In fact, this equation comes from the well-known Koszul formula
for the Levi-Civita connection, by observing that the Lie derivative
of the metric, along any Killing field is zero.
So we have
$$0 = \langle [J,I]  , K \rangle
+ \langle [J,K]  , I \rangle + \langle [I,K]  , J \rangle.
$$
Since $[J,I]_1 = ij -ji = 2k$, $[J,K]_1 = kj-jk = -2i$ and 
$[I,K]_1 = ki -ik = 2j$, we get
$0 = 2t B(k,k) -4 B(i,i) + 2s B(j,j)$.
Since $B(i,i)= B(j,j) = B(k,k)\neq 0$, this implies
$s + t = 2$.
Conversely,  if 
$s + t = 2$, we obtain by a direct calculation that
$(\nabla I)_1 = 0$.
We conclude that,   
$\langle \cdot , \cdot \rangle _s$, $0<s \leq 1$, are the 
$\Spin(3)$-invariant Riemannian metrics  on $\Spin (3)$ such that the Killing
field $I$ 
is parallel at $1$. So the index of symmetry is at least $1$. 

\begin {rem}\label {xxx}
(i) The manifold $M = (\Spin (3), \langle \cdot , \cdot \rangle _s)$ is not a 
product. Otherwise, it would split off a line. 
Assume that $0 < s < 1$. 
Then, if the index 
of symmetry is greater than $1$, by Theorem \ref {mainco}, $M$ would 
be symmetric. A direct computation shows that $(\nabla _J J)_1 = 0$. 
So $x \mapsto \text e ^{jx}$ is a closed geodesic 
of $M$ with  period  
 $2\pi \sqrt s$. This period is different from the period 
$2\pi\sqrt 2 $ of the geodesic $x \mapsto \text e ^{ix}$ 
(recall that $\langle i , i \rangle = 2$ and that $s<1$).
Then $M$ is not symmetric. Otherwise it must be 
isometric to a sphere and hence all geodesics would have the same 
length. So the index of symmetry of $M$ is $1$. 

(ii)
Let $S^2 = \Spin (3)/S^1$ be the quotient of $M = (\Spin (3), \langle \cdot ,
\cdot \rangle _s)$
by the leaves of symmetry, where 
$S^1 = \{\text {e}^{xi}: x\in \mathbb R\}$. 
It   is not difficult 
to show  that   the projection 
$\pi : (\Spin (3), \langle \, , \, \rangle _s) \to S^2= \Spin (3)/S^1$ 
is  a Riemannian submersion 
(eventually after rescaling the metric of $S^2$) if and only 
if $s=1$ (and so $t=1$). 
 Assume that 
 the full (connected) isometry group $I^o (M)$ 
of $M$ with any left-invariant Riemannian metric with $k = 2$
satisfies $\dim(I^o(M)) > 3$. 
The compact group $I^o (M)$ 
 acts on the quotient space 
$S^2$ (since any isometry preserves the foliation of symmetry). 
Then, if $S^2$ has the normal homogeneous metric, 
$I^o (M)$ 
 acts by isometries 
and thus $I^o (M)$ 
 must have  a normal subgroup of positive dimension which acts 
trivially on $S^2$. 
If $X\neq 0$ belongs to the Lie algebra of 
this normal subgroup, then $X$ defines a Killing field on $M$ which 
must be tangent to the $1$-dimensional distribution of symmetry $\mathfrak s$. 
This implies that for any two points $p,q$ in a leaf of symmetry
 there exists  $h \in I^o (M)$ with $h(p) = q$  
and such that $h$  projects trivially to the quotient $S^2$. 
Then the projection $\pi : M \to S^2$  must be a Riemannian submersion 
(for some 
$\Spin (3)$-invariant metric on $S^2$, which is unique up to scaling).  
This implies $s=t=1$. 
\end {rem}

Assume that $\Spin (3)$ together with a left-invariant Riemannian metric has
index 
of symmetry equal to $1$.
If there exists a point $g \in \Spin(3)$ such that $Z\in \so (3)$ 
is tangent to the 
$1$-dimensional leaf of symmetry $L(g)$ of $M$ at $g$, then it must always be
tangent to $L(g)$ (since the distribution of 
symmetry is invariant under isometries). 
This implies $L(g)= \Exp (tZ)(g)$ ($t \in \mathbb R$), and so $L(g)$ is 
closed 
(since all the $1$-parameter subgroups of $\Spin (3)$ are closed).

In order to describe all left-invariant Riemannian metrics on $M = \Spin (3)$ 
it only remains to analyze the case where there is
no parallel Killing field at $1$ which belongs to $\so (3)$.
This implies that $\dim I^o(M)=4$. 
In fact,  
observe that the dimension of the full isotropy group has to 
be $1$, $2$ or $3$. In the last case $M$ must is a round sphere and hence
symmetric.
The dimension of the isotropy group  
at $p \in  M$ cannot be $2$ because it would, via the isotropy representation, 
be an abelian 
$2$-dimensional subgroup of $\SO (T_p(M))\simeq \SO (3)$. Thus the dimension of
the full isotropy group must be $1$.

 In this case
there exists a non-trivial ideal $\mathfrak a$ 
of the Lie algebra $\mathfrak g$ of $G = I^o(M)$. Such an ideal must have
dimension 
$1$. In fact, this ideal must be complementary to  
$\so (3)$, which must be also an ideal, since it has codimension 
$1$ (and $\mathfrak g$ admits a bi-invariant metric). 
 Moreover, since 
any $X\in \mathfrak a$ 
projects trivially to the quotient of $M$ over the 
leaves of symmetry, $X$ must always be tangent to $\mathfrak s$. 
Observe that $X$ must be 
a left-invariant vector field since $X$ commutes with $\so (3)$. 
So, as previously observed, we may 
assume that $X=\hat i$, the left-invariant vector field with initial 
condition $i$ at $1\in \Spin (3)$ (i.e. $X_g = gi$). 
Recall that a Killing field associated with an element in $\so (3)$ 
may be regarded as a right-invariant vector field. In particular, $I$ is a 
right-invariant vector field ($I_g= ig$). Then the 
left-invariant  Rimannian metric $\langle \cdot , \cdot\rangle$ 
of $M=\Spin (3)$ is $\Ad (\Exp (ti))$-invariant. 
This implies that $i$ is an eigenvector of $A$ at $1$ and that 
the eigenvalues of  $A$ in the 
orthogonal complement of $i$ are equal, where 
$ \langle x, y \rangle = B(Ax,y)$.

So the left-invariant Riemannian metric must be associated to a  triple of 
numbers $(t,t,a)$ corresponding to the eigenvalues 
associated to the eigenvectors $j, k$ and $i$, respectively. By rescaling the
metric 
we may assume that $a=2$ (in order to be coherent with the 
first family of metrics $\langle \cdot , \cdot \rangle _s$). 
Conversely, a metric described by such a triple 
$(t,t,2)$ has a parallel Killing field at $1$.
In fact, consider the two Killing fields 
$\hat i$ and $I$, which cannot be proportional, 
because no vector field of $\Spin (3)$ can be  both 
left- and right-invariant. 
Since  the integral curves of both Killing fields  
coincide at $1$ and give a geodesic, we have 
$\nabla _i\hat i = 0 = \nabla _i\hat I$. 
Then the skew-symmetric endomorphisms $(\nabla \hat i)_1$ and $(\nabla I)_1$ of
$T_1M$
must be proportional (since $\dim (M)=3$). 
Thus there is a linear combination $\alpha \hat i 
+ \beta I$ which is parallel at $1$ (and it is non-zero, 
since $\hat i$ and $I$ are not proportional). 
Observe that when  $t=1$, $I$ is parallel at $1$ and 
so $\alpha = 0$ (the associated  metric is the same 
as
$\langle \cdot , \cdot \rangle_1$, previously described).
If $t\neq 2$, then $M$ cannot be symmetric, since 
the integral curves of $I$ and $J$, starting at $1$, 
have different length. 
In the case that $t=2$, then $\Spin (3)$ has the bi-invariant
Riemannian metric and so it is a symmetric space. 
We denote the left-invariant Riemannian metrics associated to $(t,t, 2)$ 
by $\langle \cdot , \cdot \rangle^t$, $0<t\neq 2$. 

\begin {rem}
(i) Any homothety between two different metrics
in the union of the families $\langle \cdot , \cdot \rangle_s$, $0<s < 1$, and
$\langle \cdot , \cdot \rangle^t$, $0<t \neq 2$ must be an isometry, 
since the length of the respective circles of symmetry are equal 
to $2\pi \sqrt 2$. 

(ii) No metric $\langle \cdot , \cdot \rangle_s$, $0<s<1$,
is isometric  
to a metric $\langle \cdot , \cdot \rangle ^t$, $0<t$. In fact, 
the first family 
of metrics never define a Riemannian submersion onto 
$S^2$, the quotient of $M$ 
by the leaves of symmetry, whereas  the second family always does.

(iii) Let $M_s= (\Spin (3),  \langle \cdot , \cdot \rangle_s)$.  Then, 
from Remark \ref {xxx} (ii), $I^o (M_s) = \Spin (3)$ ($0<s<1$). 
Observe that $s< 2-s < 2$ are the eigenvalues of the 
symmetric tensor $A_s$ that  
relates $\langle \cdot , \cdot\rangle_s$ with 
$\langle \cdot , \cdot \rangle = -B$, 
where $B$ is the Killing form of $\so (3)$. If $h: M_s\to M_{s'}$ 
is an  isometry, then $h$ induces a group isomorphism 
from  $\Spin (3) = I^o(M_s)$ 
onto $\Spin (3) = I^o(M_{s'})$. This implies that 
the eigenvalues of $A_s$ are the same as those of 
$A_{s'}$ and hence $s=s'$. 

(iv) If $t\neq t'$, then $\langle \cdot , \cdot \rangle ^t$ is not 
isometric to $\langle \cdot , \cdot \rangle ^{t'}$. In fact, $t/2$ 
is the radius of the sphere, obtained as the quotient of $M$ by the 
leaves of symmetry, such that the projection is a Riemannian submersion. 
\end {rem}

The previous remark finishes  the proof of Theorem \ref {dos}.

\section{Examples from fibre bundles over polars}

In this section we review the construction of certain fibre bundles by Nagano and Tanaka \cite{NaganoTanaka}, 
and show how to get examples of compact simply connected Riemannian homogeneous
manifolds with non-trivial index of symmetry.

Let $M = G/K$ be an irreducible simply connected symmetric space of compact type
and choose $o \in M$ such that $K \cdot o = o$. 
 Let $B \neq \{o\}$ be a connected 
component of 
the set of fixed points of $\sigma_o$, where $\sigma_o$ is the 
geodesic symmetry of $M$ at $o$. Note that $B$ is a totally geodesic 
submanifold, since it is a connected component of the fixed 
point set of an isometry. 
There always exists such a totally geodesic submanifold $B$ since the midpoint 
of a closed geodesic through $o$ is fixed by $\sigma_o$.

Let $d$ be the distance between $o$ and $B$ and choose $q \in B$ such that $d$
is the distance from $o$ to $q$ is equal to the distance from $o$ to $B$. 
Let $\gamma$ be a unit speed geodesic through $o$ 
and $q$ such that $\gamma(0) = o$ and $\gamma(d) = q$. Then $\gamma$ is a closed
geodesic of
period $2d$. In fact, $q = \gamma(d) = \sigma_o(\gamma(d)) = 
\gamma(-d)$. It then follows from Remark \ref{Wolf} that 
$\gamma$ is a closed geodesic.  This implies that $o$ is 
fixed by $\sigma_q$, the symmetry at $q$. Also, the 
symmetries $\sigma_o$ and $\sigma_q$ commute, since 
they both fix $o$ and their differentials commute.  

Since $M$ is simply connected, the isotropy group $K$ is connected.
One can show that $B = K \cdot q$. 
In particular, all the points in $B$ are equidistant 
to $o$. In fact,  $d_q\sigma_o$ is the 
identity when restricted to $T_qB$ and minus the identity 
when restricted to $(T_qB)^\bot$. Moreover, this holds at 
any point of $B$. So any $g \in G$ which leaves $B$ 
invariant commutes with $\sigma_o$. 
Conversely, it is obvious that $K$ maps fixed points of 
$\sigma_o$ into fixed points of $\sigma_o$. 
We thus have proved that the subgroup of 
$G$ which leaves $B$ invariant coincides with $K$.

Note that the involution $\sigma_q$ leaves $B$ 
invariant (since $B$ is totally geodesic), and so it 
maps $K$ into $K$. Thus, $(K, K^+)$ is a symmetric pair, 
where $K^+$ is the isotropy group of $K$ at $q$. 
Moreover, one has that $K^+ = K \cap K'$, where $K'$ 
is the isotropy group of $G$ at $q$. Such a symmetric 
pair is not, in general, effective 
(as one can see from the tables in \cite{NaganoTanaka}).

The totally geodesic submanifold $B$ is called a \emph{polar} of $M$.
The normal space to $T_qB$ at $q$ is a Lie triple system and hence induces,
via the exponential map, a totally geodesic 
submanifold of $M$ which is called a \emph{meridian}. 
This follows from the fact that $\exp_q((T_qB)^\bot)$ 
coincides with the set of fixed points 
of $\sigma_q \circ \sigma_o$ 
(connected component through $q$). 
In fact, if $w \in (T_qB)^\bot$ and $\beta(t)$ 
is a geodesic with $\beta'(0) = w$, then 
$(\sigma_q \circ \sigma_o)  (\beta(t)) 
= \beta(t)$, since $d_q(\sigma_q \circ \sigma_o)$ 
is the identity when restricted to $(T_qB)^\bot$. 
This shows that $\exp_q((T_qB)^\bot) $ is contained in the fixed point set of 
$\sigma_q \circ \sigma_o$. 
The other inclusion holds since $q$ is an isolated 
fixed point of $\sigma_q$.

We construct now the so-called centrioles. 
Let $p$ be the midpoint of the geodesic $\gamma$ 
joining $o$ and $q$. In line with our notation above we have
$p = \gamma(d/2)$. 
The \emph{centriole} through $p$ is the orbit $K^+ \cdot p$. 
Such an orbit is totally geodesic. In fact, the 
symmetry $\sigma_p$ interchanges $o$ and $q$, 
and so $K$ with $K'$. So $\sigma_p$ leaves 
$K^+ = K \cap K'$ invariant and, since it fixes $p$, 
leaves the centriole $K^+ \cdot p$ invariant. 
Then, $\sigma_p$  leaves the second fundamental 
form of $K^+ \cdot p$ invariant, but on the other 
hand it reverses its sign. So the centriole 
$K^+ \cdot p$ must be totally geodesic. 
Moreover, it is contained in the meridian containing $q$, since 
$K^+$ commutes with both $\sigma_q$ and $\sigma_o$ and 
$\sigma _q \circ \sigma _o (p) = p$.
We have that $(K^+, K^{++})$, where $K^{++}$ is 
the isotropy subgroup of $K^+$ at $p$, is a 
symmetric pair (not effective, in general).

We now define $S = K \cdot p$, which is a fibre 
bundle over $B$ whose  fibres are the centrioles. 
In fact, since $\gamma$ is minimizing in $[0, d]$,  
$\gamma$ is the unique (unit speed) geodesic
from $o$ to $p = \gamma (d/2)$. 
So, the isotropy $K_p$ of $K$ at $p$ must fix  $\gamma$, 
since it fixes $o$ and $p$. 
Then $K\cdot q = K\cdot \gamma (d) = q$ and therefore 
$K_p\subset K^+$, which implies $K_p = K^{++}$. 
So, we get the fiber bundle 
$$K^+/K^{++} \to K/K^{++} \to K/K^+.$$
Moreover, $K \cdot p$ turns out to be 
diffeomorphic, 
via the exponential map at $o$, to the $R$-space 
$K \cdot v \subset T_oM$, where $v = \gamma'(0)$ (or equivalently,  
$K \cdot p$ is diffeomorphic to an orbit of an $s$-representation). 

The submanifold $S = K \cdot p $ has 
parallel Killing fields in any direction of the 
centriole $K^+ \cdot p$. In fact, if $\mathfrak p^+$ 
is the Cartan subspace,  associated with $(K^+, K^{++})$, 
then $\mathfrak p^+ \subset \mathfrak p$, where $\mathfrak p$ 
is the Cartan subspace associated to $(G, K)$ 
(and elements of $\mathfrak p^+$ are parallel at $p$ on $M$, 
and so on $S$ with the induced metric).
With the same arguments as in \cite[Lemma 6.2]{OlmosReggianiTamaru}, 
one can prove the following result:

\begin{thm} \label{polar}
Let $M = G/K$ be an irreducible simply connected Riemannian symmetric space of
compact type.
Assume that the polar $B = K/K^+$ is irreducible 
and that $S = K/K^{++}$, with the induced Riemannian metric,
is not a symmetric space. 
Then the co-index of symmetry of $K/K^{++}$ is equal to the dimension of  the
polar $B = K/K^+$ and the leaves of symmetry coincide with the fibers of the
fibration $K^+/K^{++} \to K/K^{++} \to K/K^+$ (which are centrioles in $M$).
\end{thm}

\begin{proof}
We 
have already proved that the centrioles are tangent 
to the distribution of symmetry $\mathfrak s$. Note that $\mathfrak s$ 
projects down to a distribution $\bar{\mathfrak s}$ 
on the symmetric space $B = K/K^+$, which must 
be $K$-invariant (since isometries preserve 
the distribution of symmetry). 
So, since $B$ is irreducible, 
we have  $\bar{\mathfrak s} = 0$ 
or $\bar{\mathfrak s} = TB$. 
However, $\bar{\mathfrak s} = TB$ 
implies $\mathfrak s = TS$, 
which cannot happen since $S$ is not a symmetric space by assumption. 
Thus we have $\bar{\mathfrak s} = 0$,
and therefore $\mathfrak s$ coincides 
with the distribution given by the tangent 
spaces to the centrioles.
\end{proof}

\begin{exa}\label{CP2}
Consider the complex projective plane $M = {\mathbb C}P^2 = SU(3)/S(U(1)U(2))$ $ = G/K$. There is
only one polar in this situation, namely 
\[B = {\mathbb C}P^1 =
S(U(1)U(2))/S(U(1)U(1)U(1)) = K/K^+ \cong U(2)/U(1)U(1).\]
 The orbit of $K$
through the midpoint of a geodesic from $o$ to a point in $B$ is a distance
sphere $S^3 = K/K^{++}  \cong U(2)/U(1)$ in ${\mathbb C}P^2$ and the fibers of
the projection $K/K^{++} \to K/K^+$ are circles $S^1 = K^+/K^{++} \cong
U(1)U(1)/U(1) \cong U(1)$. These circles are centrioles in ${\mathbb C}P^2$. The
induced metric from ${\mathbb C}P^2$ on the distance sphere $S^3$ gives a Berger
sphere and its coindex of symmetry is equal to $2$. Up to homothety, it is one of
the metrics $\langle \cdot , \cdot \rangle^t$ in our classification for $k = 2$.
By rescaling the metric on ${\mathbb C}P^2$ one obtains other metrics in this
family. The remaining Berger sphere metrics can be obtained by  considering
distance spheres in the complex hyperbolic plane ${\mathbb C}H^2 =
SU(1,2)/S(U(1)U(2))$ which are not covered by the construction method in Theorem
\ref{polar}.
\end{exa}

\bigskip
\footnotesize
\noindent\textit{Acknowledgments.}
This research was supported by Famaf-UNC, FCEIA-UNR, CIEM-Conicet, and partially supported by Foncyt.

\end{document}